\documentclass{amsart}%

\usepackage{hyperref}
\usepackage{cite}
\usepackage{amssymb}
\usepackage{amsmath}
\usepackage{amsthm}
\usepackage{amsfonts}
\usepackage{graphicx}
\usepackage{bbm}
\usepackage{color}
\usepackage[toc,page]{appendix}
\usepackage{subfigure}
\newcommand{\beq}{\begin{equation}}
\newcommand{\eeq}{\end{equation}}
\newcommand{\bea}{\begin{eqnarray}}
\newcommand{\eea}{\end{eqnarray}}
\newcommand{\beas}{\begin{eqnarray*}}
\newcommand{\eeas}{\end{eqnarray*}}

\newtheorem{theorem}{Theorem}[section]

\newtheorem{assumption}[theorem]{Assumption}
\newtheorem{definition}[theorem]{Definition}
\newtheorem{notation}[theorem]{Notation}
\newtheorem{proposition}[theorem]{Proposition}
\newtheorem{corollary}[theorem]{Corollary}
\newtheorem{lemma}[theorem]{Lemma}
\newtheorem{remark}[theorem]{Remark}
\newtheorem{example}[theorem]{Example}
\newtheorem{examples}[theorem]{Examples}
\newtheorem{foo}[theorem]{Remarks}

\newtheorem*{acknowledgement}{Acknowledgement}
\begin{document}

\title[Gradient bounds]{Gradient bounds for Kolmogorov type diffusions}

\author[Baudoin]{Fabrice Baudoin{$^{\star}$}}
\thanks{\footnotemark {$\star$} Research was supported in part by NSF Grant DMS-1660031.}
\address{Department of Mathematics\\
University of Connecticut\\
Storrs, CT 06269, U.S.A.} \email{fabrice.baudoin@uconn.edu}

\author[Gordina]{Maria Gordina{$^{\dag \ddag}$}}
\thanks{\footnotemark {$\ddag$} Research was supported in part by the Simons Fellowship.}
\thanks{\footnotemark {$\dag$} Research was supported in part by NSF Grants DMS-1405169, DMS-1712427.}
\address{ Department of Mathematics\\
University of Connecticut\\
Storrs, CT 06269,  U.S.A.}
\email{maria.gordina@uconn.edu}

\author[Mariano]{Phanuel Mariano{$^{\dag}$}}
\address{Department of Mathematics\\
		Purdue University\\
		West Lafayette, IN 47907,  U.S.A.} \email{pmariano@purdue.edu}



\begin{abstract}
We study gradient bounds and other functional inequalities for the diffusion semigroup generated by Kolmogorov type operators. The focus is on two different methods: coupling techniques  and generalized $\Gamma$-calculus techniques. The advantages and drawbacks of each of these methods are discussed.
\end{abstract}

\maketitle

\tableofcontents

\renewcommand{\contentsname}{Table of Contents}

\section{Introduction}

In the last few years, there has been considerable interest in studying gradient bounds for semigroups generated by hypoelliptic diffusion operators. The motivation for such bounds comes from their potential applications to sub-Riemannian geometry (e.g. \cite{BaudoinGarofalo2017, BaudoinBonnefontGarofalo2014}), quasi-invariance of heat kernel measures in infinite dimensions (e.g.  \cite{BaudoinGordinaMelcher2013, Gordina2017}),  functional inequalities such as Poincar\'{e} and log-Sobolev type inequalities (e.g.  \cite{DriverMelcher2005, BaudoinBonnefont2012, Kuwada2010a, WangFY2016a}), and  the study of convergence to equilibrium for hypocoercive diffusions (e.g.  \cite{Baudoin2017a, BaudoinTardif2018}). In particular, the gradient bounds we present in this paper might be used to prove the existence of a spectral gap similarly to \cite{BaudoinWang2014a} once one has spectral localization tools. In the present paper we are interested in gradient bounds for Kolmogorov type diffusion operators for which we present and compare two different techniques: $\Gamma$-calculus methods and coupling techniques.

The Kolmogorov operator on $\mathbb{R}^2$ defined as $L=\frac{1}{2}\frac{\partial^2}{\partial x^2}+x\frac{\partial}{\partial y}$  was initially introduced by A. N.~Kolmogorov in \cite{Kolmogorov1934}, where he obtained an explicit expression for the transition density of the diffusion process whose generator is the operator $L$. Later L.~H\"ormander in \cite{Hormander1967a}  used this operator as the simplest example of a hypoelliptic second order differential operator. The semigroup generated by $L$ is Gaussian and thus the corresponding heat kernel may be computed explicitly, as was observed already by A. N.~Kolmogorov. However, despite an explicit Gaussian heat kernel, it is somehow challenging to derive relevant functional inequalities for this semigroup. We refer for instance to R. Hamilton's notes \cite{Hamilton2011}, where Riccati type equations are used to prove Li-Yau and parabolic Harnack inequalities. This (classical) Kolmogorov operator is the starting point for our consideration of several hypoelliptic operators.

As we mention above we present two techniques to prove gradient estimates in this setting.  The first technique is based on a version of the Bakry-\'Emery $\Gamma$-calculus originally introduced in \cite{BakryEmery1985}. Recall that the generator of the Kolmogorov diffusion is hypoelliptic but not elliptic, and therefore  we can not rely on the curvature-dimension inequalities to prove the gradient bounds as presented in \cite[Section 3.2.3]{BakryGentilLedouxBook}. The main novelty in our work is to introduce  a family of operators $\Gamma^{\alpha,\beta}$ and $\Gamma_2^{\alpha,\beta}$ in Notation \ref{notation 1} for which one is able to reproduce the Bakry-\'Emery arguments with some modifications in this degenerate setting. We note that over the last years versions of the $\Gamma$-calculus  have already been used for hypoelliptic operators (e.g. \cite{ BaudoinBonnefont2012, BaudoinBonnefontGarofalo2014, BaudoinGarofalo2017}). However, those references consider hypoelliptic operators satisfying the strong H\"ormander's condition and the methods developed there do not apply to the Kolmogorov operator. While \cite{Baudoin2017a} considers hypoelliptic operators satisfying the weak H\"ormander's  condition, it is mostly  concerned with hypocoercive estimates. The key idea in our work is that to get sharp gradient estimates, we make the parameters $\alpha$ and $\beta$ time-dependent as observed in Remark \ref{time-dependent BE} and the proof of Theorem \ref{Kolmogorov1}.

The second technique is coupling. The coupling techniques have seen recent progress for such degenerate operators. In \cite{BenArousCranstonKendall1995}, the authors were the first to consider couplings of hypoelliptic diffusions, as they prove existence of successful coupling for the Kolmogorov diffusion and Brownian motion on the Heisenberg group. Then in \cite{BanerjeeKendall2016a}, S.~Banerjee and W.~Kendall used a non-Markovian strategy to couple the iterated Kolmogorov diffusion. The most relevant to our results is \cite{BanerjeeGordinaMariano2018}, where coupling techniques have been used to prove gradient estimates on the Heisenberg group considered as a sub-Riemannian manifold.

The paper is organized as follows. We start by considering Kolmogorov diffusions in Section \ref{s.KolmogorovEuclid}, where we use both  generalized $\Gamma$-calculus and coupling techniques to prove gradient estimates such as in Proposition \ref{Kolmogorov1} and Proposition \ref{EuclideanCouplingBakryEmery}. This setting provides the first illustration to contrast these two methods: while the coupling method is somewhat simpler, and yields a family of gradient estimates, other functional inequalities such as the reverse Poincar\'e and the reverse log-Sobolev inequalities for the corresponding semigroup do not seem to be trackable by coupling techniques. But we can prove these inequalities by using the generalized $\Gamma$-calculus. Moreover, we are able to use only this approach (not the coupling techniques) to obtain sharper gradient bounds for the relativistic diffusion considered in Section \ref{s.RelativDiff}. The relativistic diffusion has been introduced by R. Dudley and studied extensively in \cite{Bailleul2008, Angst2015, Dudley1966, Dudley1967, Dudley1973a, DunkelHanggi2009, FranchiLeJan2007, FranchiLeJanBook2012, IkedaMatsumoto2015, JansonsMetcalfe2007, McKean1963a}. We refer the reader to\cite{DunkelPhD2008} for the history of related objects both in mathematics and physics.

In Section \ref{5.2} we use the coupling by parallel translation on Riemannian manifolds. The coupling can be described by a central limit theorem argument for the geodesic random walks as in \cite{vonRenesse2004a}.  It would be interesting to see if such a coupling can be carried out on sub-Riemannian manifolds using the approximation of Brownian motion by random walks as studied in \cite{BoscainNeelRizzi2017, AgrachevBoscainNeelRizzi2018, GordinaLaetsch2017}. If such a coupling can be constructed, then our results and techniques would be valid for an even larger class of hypoelliptic diffusions.

\section{Kolmogorov diffusion in $\mathbb{R}^d \times \mathbb{R}^d$}\label{s.KolmogorovEuclid}

Our main object in this section is a Kolmogorov diffusion in $\mathbb{R}^d \times \mathbb{R}^d$ defined  by
\[
\mathbf{X}_{t}=\left(B_t,\int_0^tB_s ds\right),
\]
where $B_{t}$ is a Brownian motion in $\mathbb{R}^d$ with the variance $\sigma^2$.

\begin{definition}
Let $f \left( p, \xi\right), p \in \mathbb{R}^d, \xi \in \mathbb{R}^{d}$ be a function on  $\mathbb{R}^d \times \mathbb{R}^{d}$. For $\sigma >0$, the Kolmogorov operator for $f \in C^{2}\left( \mathbb{R}^d \times \mathbb{R}^{d} \right)$ is defined by
\begin{align*}
& \left( Lf \right) \left( p, \xi \right):= \langle p,  \nabla_\xi f \left( p, \xi\right) \rangle+\frac{\sigma^2}{2} \Delta_pf \left( p, \xi \right)= \sum_{j=1}^{d} p_{j}\frac{\partial f}{\partial \xi_{j}}\left( p, \xi \right)+\frac{\sigma^2}{2} \Delta_p f \left( p, \xi \right),
\end{align*}
where $\Delta_p$ is the Laplace operator $\Delta$ on $\mathbb{R}^d$ acting on the variable $p$ and $\nabla_\xi $ is the gradient on $\mathbb{R}^d$ acting on the variable $\xi$.
\end{definition}
Note that for $d=1$ and $\sigma=1$ this is the original Kolmogorov operator. By H\"ormander's theorem in \cite{Hormander1967a}, the operator $L$ is hypoelliptic and generates a Markov process $X_{t}$. It follows then that the process $X_{t}$ admits a smooth transition probability density with respect to the Lebesgue measure.

\subsection{$\Gamma$-calculus}\label{EuclidGamma}

First we use  geometric methods such as generalized $\Gamma$-calculus to prove gradient bounds for the semigroup generated by the Kolmogorov operator $L$. Moreover, we show that  the estimate is sharp. We point out that a generalization of  $\Gamma$-calculus  for the Kolmogorov operator has been carried by F.Y.~Wang in \cite[pp. 300-303]{WangFYBook2014}.  However, our methods are  different and yield optimal results as we explain in Remark \ref{r.SharpGradientEst}.

Recall that the \emph{carr\'e du champ operator} for $L$ is defined by

\[
\Gamma\left( f \right):=\frac{1}{2}Lf^2 -fLf,
\]
where $f$ is from an appropriate space of functions which will be specified later. A straightforward computation shows that
\begin{equation}\label{Gamma}
\Gamma(f)= \frac{1}{2} \sigma^2\| \nabla_p f \|^2,
\end{equation}
where $\nabla_p$ is the standard gradient operator on $\mathbb{R}^d$ acting on the variable $p$, and $\Vert \cdot \Vert$ is the $\mathbb{R}^d$-norm.

\begin{notation}\label{notation 1}
For $\alpha \in \mathbb{R}$, $\beta \geqslant  0$ we define a symmetric first-order differential bilinear form $\Gamma^{\alpha,\beta}: C^{\infty}\left( \mathbb{R}^d \times \mathbb{R}^{d} \right) \times C^{\infty}\left( \mathbb{R}^d \times \mathbb{R}^{d} \right) \rightarrow \mathbb{R}$ by

\begin{align}
& \Gamma^{\alpha, \beta} (f, g) :=\sum_{i=1}^d \left( \frac{\partial f}{\partial p_i } -\alpha \frac{\partial f}{\partial \xi_i } \right)\left( \frac{\partial g}{\partial p_i } -\alpha \frac{\partial g}{\partial \xi_i } \right) +\beta\sum_{i=1}^d  \frac{\partial f}{\partial \xi_i } \frac{\partial g}{\partial \xi_i } \notag
\\
& = \langle \nabla_{p}f, \nabla_{p}g \rangle-\alpha \langle \nabla_{p}f, \nabla_{\xi}g \rangle -\alpha \langle \nabla_{\xi}f, \nabla_{p}g \rangle+\left( \alpha^{2} + \beta \right)\langle \nabla_{\xi}f, \nabla_{\xi}g \rangle, \label{e.GammaAB}
\end{align}
with the usual convention that  $\Gamma^{\alpha,\beta}(f):=\Gamma^{\alpha,\beta} \left( f, f \right)$. We will also consider

\begin{equation*}
\Gamma_2^{\alpha,\beta} (f)=\frac{1}{2} L \Gamma^{\alpha,\beta} (f) -\Gamma^{\alpha,\beta} (f,Lf).
\end{equation*}
\end{notation}
We start with the following key lemma.

\begin{lemma}\label{KeyLem}
For $f \in C^{\infty}\left( \mathbb{R}^d \times \mathbb{R}^{d} \right)$
\begin{align*}
\Gamma_2^{\alpha,\beta}(f)  \geqslant  \alpha \sum_{i=1}^d \left( \frac{\partial f}{\partial \xi_i }\right)^2 -\sum_{i=1}^d \frac{\partial f}{\partial \xi_i } \frac{\partial f}{\partial p_i }=\alpha \Vert \nabla_{\xi}f \Vert^{2}-\langle \nabla_{\xi}f, \nabla_{p}f\rangle.
\end{align*}
\end{lemma}

\begin{proof}
Let $\alpha\in\mathbb{R}, \beta\geqslant0$. A computation shows that
\begin{align*}
\Gamma_{2}^{\alpha,\beta}(f) & =\alpha\sum_{i=1}^{d}\left(\frac{\partial f}{\partial\xi_{i}}\right)^{2}-\sum_{i=1}^{d}\frac{\partial f}{\partial\xi_{i}}\frac{\partial f}{\partial p_{i}}\\
 & +\frac{\sigma^{2}}{2}\sum_{i=1}^{d}\sum_{j=1}^{d}\left(\frac{\partial^{2}f}{\partial p_{i}\partial p_{j}}-\alpha\frac{\partial^{2}f}{\partial p_{i}\partial\xi_{j}}\right)^{2}\\
 & +\frac{\sigma^{2}}{2}\beta\sum_{i=1}^{d}\sum_{j=1}^{d}\left(\frac{\partial^{2}f}{\partial p_{i}\partial\xi_{j}}\right)^{2}\\
 & \geqslant\alpha\sum_{i=1}^{d}\left(\frac{\partial f}{\partial\xi_{i}}\right)^{2}-\sum_{i=1}^{d}\frac{\partial f}{\partial\xi_{i}}\frac{\partial f}{\partial p_{i}}.
\end{align*}
\end{proof}

We are now in position to prove regularization properties for the semigroup $P_t=e^{tL}$. But first we have the following remark that will make the proofs in this section easier to read.

\begin{remark}\label{time-dependent BE} We will repeatedly use the following simple computation.
Suppose $\alpha\left( s\right),\beta\left( s\right) \in C^{1}\left( [0, \infty )\right)$. Then for $f \in C^{\infty}\left( \mathbb{R}^d \times \mathbb{R}^{d} \right)$

\begin{align}
& \phi^{\prime}\left( s \right)=2P_{s}\left(\Gamma_2^{\alpha(s),\beta(s)} \left( P_{t-s} f \right) \right)
\label{e.3.1}
\\
& -2\alpha^{\prime}\left( s \right)P_{s}\langle \nabla_{p}P_{t-s} f , \nabla_{\xi}P_{t-s} f \rangle +\left( 2 \alpha^{\prime}\left( s \right)\alpha\left( s \right)+\beta^{\prime}\left( s \right)\right)P_{s}\Vert \nabla_{\xi}P_{t-s} f\Vert^{2},
\notag
\end{align}
where $\phi$ is the functional
\[
\phi\left( s \right):=P_s \left( \Gamma^{\alpha (s), \beta\left( s\right) } \left( P_{t-s} f \right) \right), \quad 0 \leqslant  s \leqslant  t,
\]
\end{remark}

\begin{proposition}[Bakry-\'Emery type estimate] \label{Kolmogorov1}
Let $f \in C^1(\mathbb{R}^d \times \mathbb{R}^{d})$ be a globally Lipschitz function, then one has
\[
\| \nabla_p P_t f \|^2 \leqslant  \sum_{i=1}^d  P_t \left( \frac{\partial f}{\partial p_i } +t \frac{\partial f}{\partial \xi_i } \right)^2,
\]
and
\[
\| \nabla_\xi P_t f \|^2 \leqslant  P_t \|  \nabla_\xi f \|^2.
\]
\end{proposition}

\begin{proof}

Let $t>0$. We first assume that $f$ is smooth and rapidly decreasing. In that case, the following computations are easily justified since $P_t$ has a Gaussian kernel (see \cite[pp. 80-85]{CalinChangFurutaniIwasakiBook}). We consider then (at a given fixed point $(\xi, p)$) the functional
\[
\phi (s)=P_s ( \Gamma^{\alpha (s), \beta } (P_{t-s} f)), \quad 0 \leqslant  s \leqslant  t,
\]
where $\alpha(s)=-s$ and $\beta$ is a non-negative constant. Then by \eqref{e.3.1} and Lemma \ref{KeyLem}

\begin{align*}
\phi^{\prime} (s) & =2 P_s \left( \Gamma_2^{\alpha (s), \beta } \left( P_{t-s} f \right)+ \langle \nabla_{p}\left( P_{t-s} f \right), \nabla_{\xi}\left( P_{t-s} f \right) f \rangle + s\Vert \nabla_{\xi}\left( P_{t-s} f \right) \Vert^{2}\right)
\\
& \geqslant 2 P_s\left(\alpha (s) \Vert \nabla_{\xi}(P_{t-s} f) \Vert^{2}-\langle \nabla_{\xi}(P_{t-s} f), \nabla_{p}(P_{t-s} f)\rangle\right)
\\
& + \langle \nabla_{\xi}(P_{t-s} f), \nabla_{p}(P_{t-s} f)\rangle + s  \|  \nabla_\xi P_{t-s} f \|^2=0.
\end{align*}
Thus $\phi$ is increasing, and therefore $\phi(0) \leqslant  \phi(t)$, that is,
\[
\Gamma^{\alpha (0), \beta } (P_{t} f) \leqslant  P_t ( \Gamma^{\alpha (t), \beta } (f)).
\]
The result follows immediately by taking $\beta=0$.  Now, if $f \in C^1(\mathbb{R}^d \times \mathbb{R}^{d})$ is a Lipschitz function, then for any $s>0$, the function $P_s f$ is smooth and rapidly decreasing (again, since $P_s$ has a Gaussian kernel). Therefore, applying the inequality we have proved to $P_sf$ yields
\[
\Vert \nabla_p P_{t+s} f \Vert^2 \leqslant  \sum_{i=1}^d  P_t \left( \frac{\partial P_sf}{\partial p_i } +t \frac{\partial P_s f}{\partial \xi_i } \right)^2.
\]
Letting $s \to 0$ concludes the argument. To justify this limit one can first observe that since $f$ is Lipschitz then $P_s f\to f$ as $s\to 0$. Then one can also show that $P_sf$ is dominated by $g_s(p,\xi)=c_1\left( \left|p\right| +\left|\xi\right|+\sqrt{s} \right) +c_2$ for $0<s<1$ since $f$ is Lipschitz. A dominated convergence argument finishes the proof.
\end{proof}

\begin{remark}[Bakry-\'Emery type estimate is sharp]\label{r.SharpGradientEst}
Suppose $l$ is any linear form on $\mathbb{R}^d$, we define the function $f(p,\xi):=l(\xi)$. Note that $f$ is Lipschitz since $f$ is linear . Then for every $(p,\xi) \in \mathbb{R}^d \times \mathbb{R}^d$ and $t \geqslant  0$  we have
\[
P_t f (p, \xi) =\mathbb{E} \left( f\left(B_t+p, \xi +t p +\int_0^tB_s ds \right) \right)=l (\xi)+ t l(p).
\]
For this choice of $f$, one has $\| \nabla_p P_t f \|^2 =t^2 \| l \|^2$ and

\[
\sum_{i=1}^d  P_t \left( \frac{\partial f}{\partial p_i } +t \frac{\partial f}{\partial \xi_i } \right)^2=t^2 \| l \|^2.
\]
Similarly, for this choice of $f$,  $\| \nabla_\xi P_t f \|^2 = P_t \|  \nabla_\xi f \|^2$. So the bounds in Proposition \ref{Kolmogorov1} are sharp.
\end{remark}

\begin{proposition}[Reverse Poincar\'e inequality]
Let $f \in C^1(\mathbb{R}^d \times \mathbb{R}^{d})$ be a bounded function, then for $t>0$
\[
\sum_{i=1}^d \left( \frac{\partial P_t f}{\partial p_i } -\frac{1}{2} t \frac{\partial P_t f}{\partial \xi_i } \right)^2 +\frac{t^2}{12} \left( \frac{\partial P_t f}{\partial \xi_i }\right)^2 \leqslant  \frac{1}{\sigma^2 t} (P_tf^2 -(P_tf)^2).
\]
\end{proposition}

\begin{proof}
Let $t>0$. By using the same argument as in the previous proof, we can assume that $f$ is smooth and rapidly decreasing. We consider the functional
\[
\phi (s)=(t-s)P_s ( \Gamma^{\alpha (s), \beta (s) } (P_{t-s} f)), \quad 0 \leqslant  s \leqslant  t,
\]
where $\alpha(s)=\frac{1}{2} (t-s)$ and $\beta (s)=\frac{1}{12} (t-s)^2$. By \eqref{Gamma}, \eqref{e.GammaAB}, \eqref{e.3.1} and Lemma \ref{KeyLem} we have

\begin{align*}
\phi^{\prime}(s)  =&- P_s ( \Gamma^{\alpha (s), \beta (s) } (P_{t-s} f))
 \\
& +\left( t-s \right)P_{s}\langle \nabla_{p}P_{t-s} f, \nabla_{\xi}P_{t-s} f \rangle +\left( t-s \right)\left( 2 \alpha^{\prime}\left( s \right)\alpha\left( s \right)+\beta^{\prime}\left( s \right)\right)P_{s}\Vert \nabla_{\xi}P_{t-s} f\Vert^{2}
\\
&\geqslant  -P_s( \| \nabla_p P_{t-s} f \|^2) =  -\frac{2}{\sigma^2} P_s( \Gamma(P_{t-s} f)).
\end{align*}
Therefore, we have
\[
\phi (0) \leqslant  \frac{2}{\sigma^2} \int_0^t P_s( \Gamma(P_{t-s} f)) ds,
\]
where we used the fact that $\phi$ is positive. We now observe that
\[
\frac{2}{\sigma^2} \int_0^t P_s( \Gamma(P_{t-s} f)) ds=\frac{1}{\sigma^2} (P_tf^2 -(P_tf)^2).
\]
Therefore, we conclude
\[
t \Gamma^{\alpha (0), \beta (0) } (P_{t} f) \leqslant  \frac{1}{\sigma^2} (P_tf^2 -(P_tf)^2).
\]
\end{proof}

\begin{proposition}[Reverse log-Sobolev inequality]\label{reverse log-sob}
Let $f \in C^1(\mathbb{R}^d \times \mathbb{R}^{d})$ be a non-negative bounded function. One has for $t>0$
\[
 \sum_{i=1}^d \left( \frac{\partial \ln P_t f}{\partial p_i } -\frac{1}{2} t \frac{\partial \ln P_t f}{\partial \xi_i } \right)^2 +\frac{1}{12} t^2  \left( \frac{\partial \ln P_t f}{\partial \xi_i }\right)^2 \leqslant  \frac{2}{ \sigma^2 t P_t f } (P_t (f \ln f)  -P_tf \ln P_t f ).
\]
\end{proposition}

\begin{proof}
As before, we can assume that $f$ is smooth, non-negative and rapidly decreasing.
Let $t>0$. We consider the functional
\[
\phi (s)=(t-s)P_s ( (P_{t-s} f) \Gamma^{\alpha (s), \beta (s) } (\ln P_{t-s} f)), \quad 0 \leqslant  s \leqslant  t,
\]
where $\alpha(s)=\frac{1}{2} (t-s)$ and $\beta (s)=\frac{1}{12} (t-s)^2$.  Similarly to the previous proofs we have
\begin{align*}
\phi^{\prime}(s)  =&- P_s ( (P_{t-s} f)  \Gamma^{\alpha (s), \beta (s) } (\ln P_{t-s} f))+2(t-s)P_s ((P_{t-s} f)  \Gamma_2^{\alpha (s), \beta (s) } (\ln P_{t-s} f)) \\
    & -2(t-s)\alpha^{\prime}(s)\sum_{i=1}^d P_s \left( (P_{t-s} f) \frac{\partial \ln P_{t-s} f}{\partial \xi_i } \frac{\partial \ln P_{t-s}f}{\partial p_i } \right) \\
    & +2(t-s)\alpha(s)\alpha^{\prime}(s)P_s [(P_{t-s} f)  \|  \nabla_\xi \ln P_{t-s} f \|^2] \\
    & +(t-s) \beta^{\prime}(s) P_s [(P_{t-s} f)  \|  \nabla_\xi \ln P_{t-s} f \|^2 ]\\
   \geqslant  &-P_s( (P_{t-s} f)  \| \nabla_p \ln P_{t-s} f \|^2) =  -\frac{2}{\sigma^2} P_s( (P_{t-s} f)  \Gamma(\ln P_{t-s} f)).
\end{align*}
Therefore, we have
\[
\phi (0) \leqslant  \frac{2}{\sigma^2} \int_0^t P_s( (P_{t-s} f) \Gamma(\ln P_{t-s} f)) ds.
\]
We now observe that
\[
2 \int_0^t P_s( (P_{t-s} f) \Gamma(\ln P_{t-s} f)) ds=2 (P_t(f \ln f)  -P_tf \ln P_t f),
\]
and therefore
\[
t (P_t f )  \Gamma^{\alpha (0), \beta (0) } (\ln P_{t} f) \leqslant  \frac{2}{\sigma^2} (P_t(f \ln f)  -P_tf \ln P_t f).
\]
\end{proof}

The fact that the reverse log-Sobolev inequality implies a Wang-Harnack inequality for general Markov operators is by now well-known (see for instance \cite[Proposition 3.4]{BaudoinBonnefont2012}). We deduce therefore the following functional inequality.
\begin{theorem}[Wang-Harnack inequality]
Let $f$ be a non-negative Borel bounded function on $\mathbb{R}^d \times \mathbb{R}^{d}$. Then for every $t>0$, $(p,\xi),(p',\xi') \in \mathbb{R}^d \times \mathbb{R}^d$ and $\alpha >1$ we have
\begin{align*}
 & (P_t f)^\alpha \left( p,\xi \right)) \leqslant  C_{\alpha}\left( t, \left( p,\xi \right), \left( p^{\prime},\xi^{\prime} \right)\right)(P_t f^\alpha )(p^{\prime},\xi^{\prime}),
\end{align*}
where

\begin{align*}
 &
C_{\alpha}\left( t, \left( p,\xi \right), \left( p^{\prime},\xi^{\prime} \right)\right)
\\
& :=\exp \left( \frac{\alpha}{\alpha-1} \left(  \frac{6}{\sigma^2t^3} \sum_{i=1}^d \left( \frac{t}{2}(p_i'-p_i) +(\xi_i'-\xi_i)\right)^2+\frac{1}{2\sigma^2 t} \sum_{i=1}^d (p_i'-p_i)^2\right)  \right).
\end{align*}
\end{theorem}

\begin{proof}
As before we assume that $f$ is non-negative and rapidly decreasing. Let $t>0$ be fixed and $(p,\xi),(p',\xi') \in \mathbb{R}^d \times \mathbb{R}^d$. We observe first that the reverse log-Sobolev inequality in Proposition \ref{reverse log-sob} can be rewritten
\[
\Gamma^{\frac{1}{2}t, \frac{1}{12}t^2} (\ln P_t f)\leqslant   \frac{2}{ t\sigma^2 P_t f } (P_t (f \ln f)  -P_tf \ln P_t f ).
\]
We can now integrate the previous inequality as in \cite[Proposition 3.4]{BaudoinBonnefont2012} and deduce
\[
(P_t f)^\alpha (p,\xi) \leqslant  (P_t f^\alpha )(p',\xi') \exp \left( \frac{\alpha}{\alpha-1} \frac{d_t^2((p,\xi),(p',\xi'))}{2\sigma^2t}  \right).
\]
where $d_t$ is the control distance associated to the gradient $\Gamma^{\frac{t}{2}, \frac{t^2}{12}}$ defined by \eqref{e.GammaAB}. Therefore
\begin{align*}
 & d_t^2((p,\xi),(p',\xi')) \\
=& \frac{12}{t^2} \sum_{i=1}^d \left( \frac{1}{2} t(p_i'-p_i) +(\xi_i'-\xi_i)\right)^2+\sum_{i=1}^d (p_i'-p_i)^2 \\
 =&4\sum_{i=1}^d (p_i'-p_i)^2+\frac{12}{t}\sum_{i=1}^d (p_i'-p_i)(\xi'_i-\xi_i)+\frac{12}{t^2}\sum_{i=1}^d (\xi_i'-\xi_i)^2
\end{align*}
and the proof is complete.
\end{proof}

\subsection{Coupling}\label{EuclidCoupling}

In this section, we use coupling techniques to prove Proposition \ref{Kolmogorov1} under slightly different assumptions. We start by recalling the notion of a coupling. Suppose  $\left(\Omega, \mathcal{F}, \mathbb{P}\right)$  is a probability space, and  $X_{t}$ and $\widetilde{X}_{t}$ are two diffusions in $\mathbb{R}^d$ defined on this space with the same generator $L$, starting at $x, \widetilde{x} \in \mathbb{R}^d$ respectively. By their coupling we understand a diffusion $\left( X_t,\widetilde{X}_t \right)$ in $\mathbb{R}^d \times \mathbb{R}^d $ such that its law is a coupling of the laws of $X_{t}$ and $\widetilde{X}_{t}$. That is, the first and the second $d$-dimensional (marginal) distributions of $\left( X_t,\widetilde{X}_t \right)$ are given by distributions of $X_{t}$ and $\widetilde{X}_{t}$.

Let $\mathbb{P}^{\left(x,\widetilde{x}\right)}$ be the distribution of $\left( X_t,\widetilde{X}_t \right)$, so that  $\mathbb{P}^{(x,\widetilde{x})}\left( X_{0}=x,\widetilde{X}_{0}=\widetilde{x}\right)=1$. We denote by
$\mathbb{E}^{\left(x,\widetilde{x}\right)}$ the expectation with respect to the probability measure $\mathbb{P}^{\left(x,\widetilde{x}\right)}$.

To prove Proposition \ref{EuclideanCouplingBakryEmery}, we use the synchronous coupling of Brownian motions in $\mathbb{R}^d$. That is, for $\left(p,\widetilde{p}\right)\in \mathbb{R}^d\times\mathbb{R}^d$ we let $B_{t}^{p}=p+B_{t}$ and $\widetilde{B}_{t}^{\widetilde{p}}=\widetilde{p}+B_{t}$, where
$B_{t}$ is a standard Brownian motion in $\mathbb{R}^d$.

\begin{proposition}\label{EuclideanCouplingBakryEmery}
Let $f\in C^{2}\left(\mathbb{R}^{d}\times\mathbb{R}^{d}\right)$
with bounded second derivatives. If $1\leqslant q<\infty$ then for $t \geqslant 0$,
\[
\left\Vert \nabla_{p}P_{t}f\right\Vert ^{q}\leqslant \sum_{i=1}^{d}P_{t}\left(\left|\frac{\partial f}{\partial p_{i}}+t\frac{\partial f}{\partial\xi_{i}}\right|^{q}\right).
\]

\end{proposition}

\begin{proof}
Consider two copies of Kolmogorov diffusions
\begin{align*}
\mathbf{X}_{t} & =\left(B_{t}^{p},Y_{t}\right)=\left(p+B_{t},\xi+tp+\int_{0}^{t}B_{s}ds\right),\\
\widetilde{\mathbf{X}}_{t} & =\left(\widetilde{B_{t}^{p}},\widetilde{Y_{t}}\right)=\left(\widetilde{p}+\widetilde{B_{t}},\xi+t\widetilde{p}+\int_{0}^{t}\widetilde{B_{s}}ds\right),
\end{align*}
 where $B_{t}$ and $\widetilde{B_{t}}$ are two Brownian motions started
at $0$. Note that $\mathbf{X}_{t}$ starts at $(p,\xi)$ and $\widetilde{\mathbf{X}_{t}}$
starts at $\left(\widetilde{p},\xi\right)$. In order to construct a coupling
of $\left(\mathbf{X}_{t},\widetilde{\mathbf{X}_{t}}\right)$ it suffices to couple $\left(B_{t},\widetilde{B}_{t}\right)$.
Let us synchronously couple $\left(B_{t},\widetilde{B}_{t}\right)$ for
all time so that
\begin{align*}
\left|B_{t}^{p}-\widetilde{B_{t}^{p}}\right| & =\left|p-\widetilde{p}\right|,\\
\left|Y_{t}-\widetilde{Y_{t}}\right| & =t\left|p-\widetilde{p}\right|,
\end{align*}
for all $t\geqslant 0$. By using an estimate on the remainder $R$ of Taylor's approximation to $f$ and
the assumption that $f\in C^{2}\left(\mathbb{R}^{d}\times\mathbb{R}^{d}\right)$
has bounded second derivatives, there exists a $C_f \geqslant0$ such
that
\begin{align*}
 & \left|f\left(\mathbf{X}_{t}\right)-f\left(\widetilde{\mathbf{X}_{t}}\right)\right|\\
 & =\left|\sum_{i=1}^{d}\partial_{p_{i}}f\left(\widetilde{\mathbf{X}_{t}}\right)\left(p_{i}-\widetilde{p}_{i}\right)+\sum_{i=1}^{d}t\partial_{\xi_{i}}f\left(\widetilde{\mathbf{X}_{t}}\right)\left(p_{i}-\widetilde{p}_{i}\right)+R\left(\widetilde{\mathbf{X}_{t}}\right)\right|\\
 & \leqslant\sum_{i=1}^{d}\left|\left(\partial_{p_{i}}f\left(\widetilde{\mathbf{X}_{t}}\right)+t\partial_{\xi_{i}}f\left(\widetilde{\mathbf{X}_{t}}\right)\right)\right|\left|p-\widetilde{p}\right|+\frac{C_f}{2} d^2(1+t)^2\left|p-\widetilde{p}\right|^{2}.
\end{align*}
 Using this estimate and Jensen's inequality we see that
\begin{align*}
 & \left|P_{t}f(p,\xi)-P_{t}f\left(\widetilde{p},\xi\right)\right|=\left|\mathbb{E}^{\left(\left(p,\xi\right),\left(\widetilde{p},\xi\right)\right)}\left[f\left(\mathbf{X}_{t}\right)-f\left(\widetilde{\mathbf{X}_{t}}\right)\right]\right|\\
 & \leqslant\mathbb{E}^{\left(\left(p,\xi\right),\left(\widetilde{p},\xi\right)\right)}\left[\left|f\left(\mathbf{X}_{t}\right)-f\left(\widetilde{\mathbf{X}_{t}}\right)\right|\right]\\
 & \leqslant\sum_{i=1}^{d}\mathbb{E}^{\left(\left(p,\xi\right),\left(\widetilde{p},\xi\right)\right)}\left[\left|\left(\partial_{p_{i}}f\left(\widetilde{\mathbf{X}_{t}}\right)+t\partial_{\xi_{i}}f\left(\widetilde{\mathbf{X}_{t}}\right)\right)\right|^{q}\right]^{\frac{1}{q}}\left|p-\widetilde{p}\right|\\
 & +\frac{C_f}{2} d^2(1+t)^2\left|p-\widetilde{p}\right|^{2}\\
 & =\sum_{i=1}^{d}P_{t}\left(\left|\left(\partial_{p_{i}}f\left(\widetilde{p},\xi\right)+t\partial_{\xi_{i}}f\left(\widetilde{p},\xi\right)\right)\right|^{q}\right)^{\frac{1}{q}}\left|p-\widetilde{p}\right|\\
 & +\frac{C_f}{2} d^2(1+t)^2\left|p-\widetilde{p}\right|^{2}.
\end{align*}
Dividing out by $\left|p-\widetilde{p}\right|$ and taking $\widetilde{p}\to p$
we have that
\begin{align*}
\left\Vert \nabla_{p}P_{t}f\left(p,\xi\right)\right\Vert  & =\limsup_{\widetilde{p}\to p}\frac{\left|P_{t}f(p,\xi)-P_{t}f\left(\widetilde{p},\xi\right)\right|}{\left|p-\widetilde{p}\right|}\\
 & \leqslant \sum_{i=1}^{d}P_{t}\left(\left|\partial_{p_{i}}f\left(p,\xi\right) +t\partial_{\xi_{i}}f\left(p,\xi\right)\right|^{q}\right)^{\frac{1}{q}},
\end{align*}
which proves the statement.
\end{proof}

\begin{remark}
When $q=2$, this coincides with the conclusion of Proposition \ref{Kolmogorov1}. The coupling method here is simpler than the $\Gamma$-calculus method and moreover yields a family of inequalities for $q \geqslant 1$. However, on the other hand, it appears difficult to prove the reverse Poincar\'e and the reverse log-Sobolev inequalities for the semigroup by using coupling techniques.
\end{remark}

\section{Relativistic diffusion}\label{s.RelativDiff}

In this section we consider the diffusion $\mathbf{X}_t=(B_t,\int_0^tB_s ds)$, where $B_t$ is a Brownian motion on the $d$-dimensional hyperbolic space $\mathbb{H}^d$. This is the relativistic Brownian motion introduced by R.~Dudley \cite{Dudley1966} and studied by J.~Franchi and Y.~Le Jan in \cite{FranchiLeJan2007}. In this section, we will prove functional inequalities for the generator of $\mathbf{X}_t$. Our methods will only involve $\Gamma$-calculus through generalized curvature dimension conditions. The emphasis on $\Gamma$-calculus in this section will allow us to obtain sharper estimates for the relativistic diffusion. In particular, the estimate \eqref{RelGrad} in Corollary \ref{RelGradThm} is sharper than the ones given in Theorems \ref{GenGrad-Gamma} and  \ref{GenGradient}. In the following sections we will prove similar theorems using both $\Gamma$-calculus and coupling techniques but for a larger class of diffusions.

We follow the notation in \cite{FranchiLeJan2007}. Recall that the Minkowski space is the product $\mathbb{R} \times \mathbb{R}^{d}$  with $d \geqslant 2$
 \[
 \mathbb{R}^{1,d}=\{ \xi=(\xi_0,  \vec{\xi} ) \in \mathbb{R}\times \mathbb{R}^d ) \}
 \]
equipped with the Lorentzian norm $q\left( \xi, \xi \right):=\xi_0^2 - \Vert \vec{\xi} \Vert^2$. The standard basis in $\mathbb{R}^{1,d}$ is denoted by $e_{0}, ..., e_{d}$. Let $\mathbb{H}^d$ be the positive half of the  unit sphere in $\mathbb{R}^{1,d}$, namely,

\[
\mathbb{H}^d:=\left\{  p \in \mathbb{R}^{1,d}: p_{0}>0, q\left( p, p \right)=1 \right\}.
\]
Note that $\mathbb{H}^d$ has a standard parametrization $p=( p_{0},  \vec{p} )=\left( \cosh r, \sinh r \ \omega\right)$ with $r \geqslant 0$, $\omega \in \mathbb{S}^{d-1}$. In these coordinates the hyperbolic metric is given by $dr^{2}+\sinh^{2}r d\omega^{2}$, where $d\omega$ is the metric on the sphere $\mathbb{S}^{d-1}$, and the volume element is

\[
\int_{\mathbb{H}^d} f\left( \Omega \right) d\Omega=\int_{0}^{\infty}\int_{\mathbb{S}^{d-1}}  f\left( r, \omega \right) \sinh^{d-1}r dr d\omega.
\]
Finally, the corresponding Laplace-Beltrami operator $\mathbb{H}^d$ can be written in these coordinates as follows (see \cite[Proposition 3.5.4]{FranchiLeJanBook2012}).

\[
\Delta^\mathbb{H}f\left(  r, \omega \right):=\frac{\partial^{2}f}{\partial r^{2}}\left(   r, \omega \right)+\left( d-1 \right) \coth r \frac{\partial f}{\partial r }\left(  r, \omega \right) +\frac{1}{\sinh^{2} r } \Delta_{\mathbb{S}^{d-1}}^{\omega} f\left(  r, \omega \right),
\]
where $\Delta_{\mathbb{S}^{d-1}}^{\omega}$ is the Laplace operator on $\mathbb{S}^{d-1}$ acting on the variable $\omega$. We denote by $\nabla^\mathbb{H}$ the gradient on $\mathbb{H}^d$ viewed as a Riemannian manifold.

Following the construction in \cite{Dudley1966}, we consider a stochastic process with values in the unitary tangent bundle $T^{1}\mathbb{R}^{1,d}$ of the Minkowski space-time $ \mathbb{R}^{1,d}$. We identify the unit tangent bundle with $\mathbb{H}^d \times \mathbb{R}^{1,d}$. Then the relativistic Brownian motion is the process $X_{t}:=\left( g_{t}, \xi_{t} \right)$, where  $g_{t}$ is a Brownian motion in $\mathbb{H}^d$ starting at $e_{0}$, and the second process is the time integral of $g_{t}$

\[
\xi_{t}:=\int_{0}^{t} g_{s}ds.
\]
By \cite[Theorem VII.6.1]{FranchiLeJanBook2012} the process $X_{t}$ is a Markov Lorentz-invariant diffusion whose generator is the relativistic Laplacian defined as follows. For $\sigma >0$, the relativistic Laplacian for $f \in C^{2}\left( \mathbb{H}^d \times \mathbb{R}^{1,d} \right)$ is the operator
\begin{align*}
& \left( Lf \right) \left( p, \xi \right)= \langle p, \nabla_\xi f \left( p, \xi\right) \rangle+\frac{\sigma^2}{2} \Delta_p^\mathbb{H}f \left( p, \xi \right)=
\\
& p_{0}\frac{\partial f}{\partial \xi_{0}}\left( p, \xi \right)+\sum_{j=1}^{d} p_{j}\frac{\partial f}{\partial \xi_{j}}\left( p, \xi \right)+\frac{\sigma^2}{2} \Delta_p^\mathbb{H}f \left( p, \xi \right),
\end{align*}
where $\Delta_p^\mathbb{H}$ is the Laplace-Beltrami operator $\Delta^\mathbb{H}$ on $\mathbb{H}^d$ acting on the variable $p$. The operator $L$ is hypoelliptic and generates the Markov process $X_{t}$. Let $P_{t}$ be the heat semigroup with the operator $L$ being its generator.

We consider functions on $\mathbb{H}^d \times \mathbb{R}^{1, d}$ with $f \left( p, \xi\right), p \in \mathbb{H}^d, \xi \in \mathbb{R}^{1,d}$. Recall that operators $\nabla^\mathbb{H}$ and $\Delta^\mathbb{H}$ act on the variable $p$ for  $f\left( p, \xi \right)$. We use $\nabla_\xi$ for the usual Euclidean gradient. Let $\Gamma\left( f \right)$ be the \emph{carr\'e du champ operator} for $L$. Recall that we view $\mathbb{H}^d$ as a Riemannian manifold with $\Delta^\mathbb{H}$ being the Laplace-Beltrami operator.

Our main result of this section is a generalized curvature-dimension inequality for $\mathbb{H}^d \times \mathbb{R}^{1, d}$ with the operator $L$ and $\nabla_{\xi}$ playing a role of the vertical gradient. Namely, we define a symmetric, first-order differential bilinear form $\Gamma^{Z}: C^{\infty}\left( \mathbb{H}^d \times \mathbb{R}^{1, d} \right) \times C^{\infty}\left( \mathbb{H}^d \times \mathbb{R}^{1, d} \right) \rightarrow \mathbb{R}$ by

\begin{equation}\label{e.4.1}
\Gamma^Z(f):= \| \nabla_\xi f \|^2,
\end{equation}
for any $f \in C^{\infty}\left( \mathbb{H}^d \times \mathbb{R}^{1, d} \right)$.

\begin{theorem}[Curvature-dimension condition]\label{t.4.1} The operator $L$ satisfies the following generalized curvature-dimension condition for any $f \in C^{\infty}\left( \mathbb{H}^d \times \mathbb{R}^{1, d} \right)$

\begin{align*}
& \Gamma_2(f)\geqslant -\frac{d}{2}\sigma^{2} \Gamma(f)-\frac{1}{4} \Gamma^Z(f),
\\
& \Gamma_2^Z (f) \geqslant 0.
\end{align*}

\end{theorem}

\begin{proof}
A simple calculation of the carr\'e du champ operator for $L$ is given by
\[
\Gamma(f):=\frac{1}{2} (Lf^2-2fLf)=\frac{\sigma^2}{2} \| \nabla^\mathbb{H}_p f \|^2,
\]
where as before $\nabla^\mathbb{H}_p$ is the Riemannian gradient on $\mathbb{H}^d$. Straightforward  computations show that the iterated carr\'e du champ operator
\[
\Gamma_2(f):= \frac{1}{2} (L\Gamma(f)-2\Gamma(f,Lf))
\]
is given by
\[
\Gamma_2(f)= \frac{\sigma^4}{4} \Gamma^\mathbb{H}_2(f)-\frac{\sigma^2}{2} \langle \nabla^\mathbb{H}_p f, \nabla_\xi f \rangle,
\]
where $\Gamma^\mathbb{H}_2(f)$ is the iterated carr\'e du champ operator for $\Delta_p^\mathbb{H}$. Recall that we view $\mathbb{H}^d$ as a Riemannian manifold with $\Delta^\mathbb{H}$ being the Laplace-Beltrami operator, therefore we can use Bochner's formula for $\Delta_p^\mathbb{H}$
\[
\Gamma^\mathbb{H}_2(f) \geqslant -(d-1)\| \nabla^\mathbb{H}_p f \|^2,
\]
thus
\[
\Gamma_2(f)\geqslant -\frac{d-1}{2} \sigma^{2} \Gamma(f)-\frac{\sigma^2}{2} \langle \nabla^\mathbb{H}_p f, \nabla_\xi f \rangle.
\]
Now we can use an elementary estimate

\[
-\frac{\sigma^2}{2} \langle \nabla^\mathbb{H}_p f, \nabla_\xi f \rangle \geqslant -\frac{\sigma^4}{4} \Vert \nabla^\mathbb{H}_p f \Vert^{2}-\frac{1}{4} \Vert \nabla_\xi f \Vert^{2}=-\frac{\sigma^2}{2} \Gamma\left( f \right)-\frac{1}{4} \Vert \nabla_\xi f \Vert^{2}
\]
to see that

\[
\Gamma_2(f)\geqslant -\frac{d}{2}\sigma^{2} \Gamma(f)-\frac{1}{4} \Vert \nabla_\xi f \Vert^{2}.
\]
The last term in this inequality is the bilinear form $\Gamma^Z$ defined by \eqref{e.4.1}. Its iterated form is

\[
\Gamma^Z_2(f):= \frac{1}{2} (L\Gamma^Z(f)-2\Gamma^Z(f,Lf)),
\]
for which another routine computation shows that
\[
\Gamma^Z_2(f)= \frac{\sigma^2}{2} \| \nabla_\xi  \nabla^\mathbb{H}_p f \|^2 \geqslant 0,
\]
which concludes the proof.
\end{proof}


For later use, our first task is to construct a convenient Lyapunov function for the operator $L$. A \emph{Lyapunov function} on $ \mathbb{H}^d \times \mathbb{R}^{1,d}$ for the operator $L$ is a smooth function $W$ such that $LW \leqslant CW$ for some $C>0$.  Consider the function
\begin{equation} \label{e.LyapunovFunction}
W(p,\xi):=1+\xi_0^2+\| \vec{\xi} \|^2+d_R (p_0,p)^2, \quad  p \in \mathbb{H}^d, \xi \in \mathbb{R}^{1,d},
\end{equation}
where $p_0$ is a fixed point in $\mathbb{H}^d$ and $d_R $ is the Riemannian distance in $ \mathbb{H}^d$.

We observe that $W$ is smooth since $d_R (p_0,\cdot)^2$ is (on the hyperbolic space the exponential map at $p_0$, is a diffeomorphism). Using the Laplacian comparison theorem on $\mathbb{H}^d$, one can see that $W$ has the following properties
\begin{align*}
& W \geqslant  1,
\\
& \| \nabla_\xi W \| + \| \nabla_p W\| \leqslant  C W,
\\
& LW \leqslant  C W \text{ for some constant } C>0,
\\
& \{ W \leqslant  m \} \text{ is compact for every } m.
\end{align*}
We shall make use of the Lyapunov function $W$ defined by \eqref{e.LyapunovFunction} to prove the following result.

\begin{theorem}[Gradient estimate] \label{t.3.2} Consider the operator $L$ and its corresponding heat semigroup $P_t$. For any $f \in C_{0}^{\infty}\left( \mathbb{H}^d \times \mathbb{R}^{1, d} \right)$ and $t \geqslant 0$

\[
2d\sigma^{2} \Gamma\left( P_{t}f \right)\left( x \right)+\Gamma^{Z}\left( P_{t}f \right)\left( x \right) \leqslant e^{d\sigma^{2}t}\left( 2d\sigma^{2} P_{t}\left(\Gamma\left( f \right) \right)\left( x \right)+ P_{t}\left(\Gamma^{Z}\left( f \right) \right)\left( x \right) \right).
\]
\end{theorem}

\begin{proof} We fix $t > 0$ throughout the proof. For $0 < s < t$, $x \in \mathbb{H}^d \times \mathbb{R}^{1, d}$  we denote

\begin{align*}
& \varphi_{1}\left( x, s \right):=\Gamma\left( P_{t-s}f \right)\left( x \right),
\\
&  \varphi_{2}\left( x, s \right):=\Gamma^{Z}\left( P_{t-s}f \right)\left( x \right).
\end{align*}
Then
\begin{align*}
& L\varphi_{1}+\frac{\partial \varphi_{1}}{\partial s}=2\Gamma_{2}\left( P_{t-s}f \right),
\\
&  L\varphi_{2}+\frac{\partial \varphi_{2}}{\partial s}=2\Gamma_{2}^{Z}\left( P_{t-s}f \right).
\end{align*}
Now we would like to find two non-negative smooth functions $a\left( s \right)$ and $b\left( s \right)$ such that for
\[
\varphi\left( x, s \right):= a\left( s \right)\varphi_{1}\left( x, s \right) +b\left( s \right) \varphi_{2}\left( x, s \right),
\]
we have
\[
L\varphi+\frac{\partial \varphi}{\partial s}  \geqslant 0.
\]
Then by Theorem \ref{t.4.1} we have

\begin{align*}
& L\varphi+\frac{\partial \varphi}{\partial s}=
\\
& a^{\prime}\left( s \right) \Gamma\left( P_{t-s}f \right)+b^{\prime}\left( s \right) \Gamma^{Z}\left( P_{t-s}f \right)+2a\left( s \right)\Gamma_{2}\left( P_{t-s}f \right)+2b\left( s \right)\Gamma_{2}^{Z}\left( P_{t-s}f \right)\geqslant
\\
& a^{\prime}\left( s \right) \Gamma\left( P_{t-s}f \right)+b^{\prime}\left( s \right) \Gamma^{Z}\left( P_{t-s}f \right)+2a\left( s \right) \left( -\frac{d}{2}\sigma^{2} \Gamma\left( P_{t-s}f \right)-\frac{1}{4} \Gamma^Z\left( P_{t-s}f \right)\right)=
\\
& \left( a^{\prime}-ad\sigma^{2}\right)\Gamma\left( P_{t-s}f \right)+\left( b^{\prime}-\frac{a}{2}\right)\Gamma^Z\left( P_{t-s}f \right).
\end{align*}
One can easily see that if we choose $b\left( s \right)=e^{\alpha s}$ and $a\left( s \right)=ke^{\alpha s}$ with $\alpha=d\sigma^{2}$ and $k=2d\sigma^{2}$, then the last expression is $0$. Using the existence of the Lyapunov function $W$ as defined by \eqref{e.LyapunovFunction} and a cutoff argument as in \cite[Theorem 7.3]{BaudoinEMS2014}, we deduce from a parabolic comparison principle
\[
P_{t}\left( \varphi\left( \cdot, t \right) \right) \left( x \right)\geqslant \varphi\left( x, 0 \right).
\]
Observe that

\begin{align*}
& \varphi\left( x, 0 \right)= a\left( 0 \right)\varphi_{1}\left( x, 0 \right) +b\left( 0 \right) \varphi_{2}\left( x, 0 \right)=2d\sigma^{2} \Gamma\left( P_{t}f \right)\left( x \right)+\Gamma^{Z}\left( P_{t}f \right)\left( x \right),
\\
& P_{t}\left( \varphi\left( \cdot, t \right) \right) \left( x \right)=a\left( t \right) P_{t}\left(\Gamma\left( f \right) \right)\left( x \right)+b\left( t \right) P_{t}\left(\Gamma^{Z}\left( f \right) \right)\left( x \right)=
\\
& e^{d\sigma^{2}t}\left( 2d\sigma^{2} P_{t}\left(\Gamma\left( f \right) \right)\left( x \right)+ P_{t}\left(\Gamma^{Z}\left( f \right) \right)\left( x \right) \right),
\end{align*}
therefore

\[
2d\sigma^{2} \Gamma\left( P_{t}f \right)\left( x \right)+\Gamma^{Z}\left( P_{t}f \right)\left( x \right) \leqslant e^{d\sigma^{2}t}\left( 2d\sigma^{2} P_{t}\left(\Gamma\left( f \right) \right)\left( x \right)+ P_{t}\left(\Gamma^{Z}\left( f \right) \right)\left( x \right) \right).
\]

\end{proof}

\begin{corollary}[Poincar\'e type inequality] For any $f\in C_{0}^{\infty}\left(\mathbb{H}^d \times \mathbb{R}^{1, d}\right)$ and $t \geqslant 0$
\[
P_{t}\left(f^{2}\right)-\left(P_{t}f\right)^{2}\leqslant \frac{e^{d\sigma^{2}t}-1}{\left(d\sigma^{2}\right)^{2}}\left(2d\sigma^{2}P_{t}\left(\Gamma\left(f\right)\right)+P_{t}\left(\Gamma^{Z}\left(f\right)\right)\right).
\]

\end{corollary}

\begin{proof}
Since $\Gamma^{Z}\left(f\right):=\left\Vert \nabla_{\xi}f\right\Vert ^{2}\geqslant 0$
and $P_{t}\left(f^{2}\right)-\left(P_{t}f\right)^{2} =2\int_{0}^{t}P_{s}\left(\Gamma\left(P_{t-s}f\right)\right)ds$,
then for $\sigma>0$,

\begin{align*}
 & \int_{0}^{t}P_{s}\left(2d\sigma^{2}\Gamma\left(P_{t-s}f\right)+\Gamma^{Z}\left(P_{t-s}f\right)\right)ds
 \\
 & \geqslant \int_{0}^{t}P_{s}\left(2d\sigma^{2}\Gamma\left(P_{t-s}f\right)\right)ds=d\sigma^{2}\left(P_{t}\left(f^{2}\right)-\left(P_{t}f\right)^{2}\right).
\end{align*}
By Theorem \ref{t.3.2} we have that
\begin{align*}
 & \int_{0}^{t}P_{s}\left(2d\sigma^{2}\Gamma\left(P_{t-s}d\right)+\Gamma^{Z}\left(P_{t-s}f\right)\right)ds
 \\
 & \leqslant \int_{0}^{t}e^{d\sigma^{2}\left(t-s\right)}P_{s}\left(2d\sigma^{2}P_{t-s}\left(\Gamma\left(f\right)\right)+P_{t-s}\left(\Gamma^{Z}\left(f\right)\right)\right)ds
 \\
 & = \left(2d\sigma^{2}P_{t}\left(\Gamma\left(f\right)\right)+P_{t}\left(\Gamma^{Z}\left(f\right)\right)\right)\int_{0}^{t}e^{d\sigma^{2}\left(t-s\right)}ds
 \\
 & = \frac{e^{d\sigma^{2}t}-1}{d\sigma^{2}}\left(2d\sigma^{2}P_{t}\left(\Gamma\left(f\right)\right)+P_{t}\left(\Gamma^{Z}\left(f\right)\right)\right).
\end{align*}
This implies
\[
P_{t}\left(f^{2}\right)-\left(P_{t}f\right)^{2}\leqslant \frac{e^{d\sigma^{2}t}-1}{\left(d\sigma^{2}\right)^{2}}\left(2d\sigma^{2}P_{t}\left(\Gamma\left(f\right)\right)+P_{t}\left(\Gamma^{Z}\left(f\right)\right)\right).
\]

\end{proof}

The next corollary gives us an equivalent estimate to the one in Theorem \ref{t.3.2}. The estimate \eqref{RelGrad} will be similar to the one we will obtain in Theorem \ref{GenGradient} in a more general setting.

\begin{corollary}\label{RelGradThm}
For any $f\in C_{0}^{\infty}\left(\mathbb{H}^d \times \mathbb{R}^{1, d}\right)$, the gradient
estimate
\[
2d\sigma^{2}\Gamma\left(P_{t}f\right)+\Gamma^{Z}\left(P_{t}f\right)\leqslant e^{d\sigma^{2}t}\left(2d\sigma^{2}P_{t}\left(\Gamma(f)\right)+P_{t}\left(\Gamma^{Z}(f)\right)\right),
\]
is equivalent to
\begin{equation}\label{RelGrad}
\Gamma(P_{t}f)\leqslant e^{d\sigma^{2}t}P_{t}\left(\Gamma\left(f\right)\right)+\frac{e^{d\sigma^{2}t}-1}{2d\sigma^{2}}P_{t}\left(\Gamma^{Z}\left(f\right)\right).
\end{equation}

Moreover, one has
\[
\Gamma^{Z}\left(P_{t}f\right)\leqslant P_{t}\left(\Gamma^{Z}\left(f\right)\right).
\]

\end{corollary}

\begin{proof}
Recall that
\[
P_{t}\left(\Gamma(f)\right)-\Gamma(P_{t}f)=2\int_{0}^{t}P_{s}\left(\Gamma_{2}\left(P_{t-s}f\right)\right)ds.
\]
 Using the curvature dimension inequality $\Gamma^{Z}\left(f\right)\geqslant -2d\sigma^{2}\Gamma(f)-4\Gamma_{2}\left(f\right)$
we have

\begin{align*}
 & \int_{0}^{t}P_{s}\left(2d\sigma^{2}\Gamma\left(P_{t-s}f\right)+\Gamma^{Z}\left(P_{t-s}f\right)\right)ds
 \\
 & \geqslant \int_{0}^{t}P_{s}\left(2d\sigma^{2}\Gamma\left(P_{t-s}f\right)-2d\sigma^{2}\Gamma(P_{t-s}f)-4\Gamma_{2} \left(P_{t-s}f\right)\right)ds
 \\
 & =-2\left(P_{t}\left(\Gamma(f)\right)-\left(\Gamma(P_{t}f)\right)\right).
\end{align*}
On the other hand we have
\begin{align*}
 & \int_{0}^{t}P_{s}\left(2d\sigma^{2}\Gamma\left(P_{t-s}f\right)+\Gamma^{Z}\left(P_{t-s}f\right)\right)ds
\\
 & \leqslant \int_{0}^{t}e^{d\sigma^{2}\left(t-s\right)}P_{s}\left(2d\sigma^{2}P_{t-s}\left(\Gamma\left(f\right)\right)
 +P_{t-s}\left(\Gamma^{Z}\left(f\right)\right)\right)ds
 \\
 & = \left(2d\sigma^{2}P_{t}\left(\Gamma\left(f\right)\right)+P_{t}\left(\Gamma^{Z}\left(f\right)\right)\right)
 \int_{0}^{t}e^{d\sigma^{2}\left(t-s\right)}ds
 \\
 & =\frac{e^{d\sigma^{2}t}-1}{d\sigma^{2}}\left(2d\sigma^{2}P_{t}\left(\Gamma\left(f\right)\right)+P_{t}\left(\Gamma^{Z}\left(f\right)\right)\right).
\end{align*}
Putting these together we have
\[
\Gamma(P_{t}f)-P_{t}\left(\Gamma(f)\right)\leqslant \frac{e^{d\sigma^{2}t}-1}{2d\sigma^{2}}\left(2d\sigma^{2}P_{t}\left(\Gamma\left(f\right)\right)+P_{t}\left(\Gamma^{Z}\left(f\right)\right)\right).
\]
A rearranging of this inequality gives us
\[
\Gamma(P_{t}f)\leqslant e^{d\sigma^{2}t}P_{t}\left(\Gamma\left(f\right)\right)+\frac{e^{d\sigma^{2}t}-1}{2d\sigma^{2}}P_{t}\left(\Gamma^{Z}\left(f\right)\right).
\]
Conversely, assume $\Gamma(P_{t}f)\leqslant e^{d\sigma^{2}t}P_{t}\left(\Gamma\left(f\right)\right)+\frac{e^{d\sigma^{2}t}-1}{2d\sigma^{2}}P_{t}\left(\Gamma^{Z}\left(f\right)\right)$
then
\begin{align*}
 & 2d\sigma^{2}\Gamma(P_{t}f)+\Gamma^{Z}\left(P_{t}f\right)
\\
 & \leqslant  2d\sigma^{2}\left(e^{d\sigma^{2}t}P_{t}\left(\Gamma\left(f\right)\right)+\frac{e^{d\sigma^{2}t}-1}{2d\sigma^{2}}P_{t}\left(\Gamma^{Z}\left(f\right)\right)\right)+\Gamma^{Z}\left(P_{t}f\right)
 \\
 & =e^{d\sigma^{2}t}\left(2d\sigma^{2}P_{t}\left(\Gamma(f)\right)+P_{t}\left(\Gamma^{Z}(f)\right)\right)+\Gamma^{Z}\left(P_{t}f\right)-P_{t}\left(\Gamma^{Z}\left(f\right)\right)
\\
 & \leqslant e^{d\sigma^{2}t}\left(2d\sigma^{2}P_{t}\left(\Gamma(f)\right)+P_{t}\left(\Gamma^{Z}(f)\right)\right)+0.
\end{align*}
 The last inequality is due to $\Gamma^{Z}\left(P_{t}f\right)\leqslant P_{t}\left(\Gamma^{Z}\left(f\right)\right)$.
To see this, consider the functional $\phi(s)=P_{s}\left(\Gamma^{Z}\left(P_{t-s}f\right)\right)$
for $0\leqslant s\leqslant t$ . A calculation shows that
\[
\Phi^{\prime}(s)=2P_{s}\left(\Gamma_{2}^{Z}\left(P_{t-s}f\right)\right)\geqslant 0,
\]
which shows $\phi(s)$ is increasing, so that $0\leqslant \phi(t)-\phi(0)=P_{t}\left(\Gamma^{Z}\left(f\right)\right)-\Gamma^{Z}\left(P_{t}f\right)$.
\end{proof}

\section{Gradient bounds for a general Kolmogorov diffusion}

We now study generalizations of Kolmogorov type diffusions and prove gradient bounds using $\Gamma$-calculus and coupling techniques. We give examples and give other generalizations along the way.

We start with an outline of this section. In Section \ref{5.1}, we prove gradient bounds for diffusions of the type $\mathbf{X}_{t}=(X_t,\int_0^t\sigma(X_s) ds)$ for $\sigma:\mathbb{R}^k\to\mathbb{R}^k$, where $X_t$ is a Markov process on $\mathbb{R}^k$ satisfying a $\Gamma_2$ lower bound. We show that a generalized curvature-dimension condition for the generator of $\mathbf{X}_{t}$ is satisfied similarly to Theorem \ref{t.4.1}. In Section \ref{5.1.1}, we show that the results in Section \ref{5.1} are valid when $X_t$ is a Brownian motion on a complete Riemannian manifold  isometrically embedded in $\mathbb{R}^{k}$ for some $k$. In Section \ref{5.2}, we use coupling techniques to prove gradient bounds when $X_t$ is assumed to live in a Riemannian manifold $M$. In this section $M$ is not necessarily embedded in some $\mathbb{R}^k$. In Section \ref{5.3}, we generalize the results in Section \ref{5.2} to iterated Kolmogorov diffusions. Finally in Section \ref{5.4}, we prove gradient bounds when $X_t$ is a hypoelliptic Brownian motion on the Heisenberg group.

 \subsection{$\Gamma$-calculus}\label{5.1}

We now study the diffusion $\mathbf{X}_{t}=\left(X_{t},\int_{0}^{t}\sigma\left(X_{s}\right)ds\right)$, where $X_{t}$ is a Markov process in $\mathbb{R}^{k}$ whose generator is given by
\[
L=\sum_{i=1}^{k}V_{i}^{2}+V_{0}.
\]
Here $V_i, i=0,\dots,k$ are smooth vector fields, and we assume that $\sigma: \mathbb{R}^{k}\to\mathbb{R}^{k}$ is a $C^{1}$ map such that

\begin{equation}\label{C-sigma1}
C_{\sigma}:= \sup_{p\in \mathbb{R}^k}\left(\sum_{i,j=1}^{d}(V_{i}\sigma_{j}(p))^{2}\right)^{\frac{1}{2}}<\infty.
\end{equation}
We consider functions on $\mathbb{R}^{k}\times\mathbb{R}^{k}$ with
$f(p,\xi)$, $p,\xi\in\mathbb{R}^{k}$. The generator for $\mathbf{X}_{t}$
is given by
\[
\mathcal{L}=L+\sum_{i=1}^{k}\sigma_{i}(p)\frac{\partial}{\partial\xi_{i}}.
\]
We first prove a generalized curvature-dimension inequality for $\mathcal{L}$
given some assumptions on $L$. Let $\Gamma(f)$ be the \emph{carr\'e du
champ} operator for $\mathcal{L}$, while $\Gamma^{L}(f)$ will be
associated with $L$. Let $\Gamma_{2}(f)$ and $\Gamma_{2}^{L}(f)$ be
the corresponding iterated \emph{carr\'e du champ} operators.

We define a symmetric, first-order differential bilinear form  $\Gamma^{Z}:C^{\infty}\left(\mathbb{R}^{k}\times\mathbb{R}^{k}\right)\times C^{\infty}\left(\mathbb{R}^{k}\times\mathbb{R}^{k}\right)\to\mathbb{R}$
by
\[
\Gamma^{Z}(f)=\left\Vert \nabla_{\xi}f\right\Vert ^{2}
\]
for any $f\in C^{\infty}\left(\mathbb{R}^{k}\times\mathbb{R}^{k}\right)$.
\begin{theorem}[Curvature-dimension inequality]
\label{Gen CD inq} If the operator
$L$ satisfies
\[
\Gamma_{2}^{L}(f)\geqslant \rho\Gamma^{L}(f),
\]
then the operator $\mathcal{L}$ satisfies the following generalized
curvature-dimension inequality for any $f\in C^{\infty}\left(\mathbb{R}^{k}\times\mathbb{R}^{k}\right)$,
\begin{align*}
\Gamma_{2}(f) & \geqslant\left(\rho-\frac{C_{\sigma}}{2}\right)\Gamma(f)-\frac{C_{\sigma}}{2}\Gamma^{Z}(f),\\
\Gamma_{2}^{Z}(f) & \geqslant 0.
\end{align*}

\end{theorem}

\begin{proof}
A simple calculation of the \emph{carr\'e du champ} of $\mathcal{L}$ and $L$
shows that
\begin{align*}
\Gamma(f) & :=\frac{1}{2}\left(\mathcal{L}f^{2}-2f\mathcal{L}f\right)=\sum_{j=1}^{k}\left(V_{i}f\right)^{2},\\
\Gamma^{L}(f) & :=\frac{1}{2}\left(Lf^{2}-2fLf\right)=\sum_{j=1}^{k}\left(V_{i}f\right)^{2}.
\end{align*}
More computations of the iterated carr\'e du champ $\Gamma_{2}(f):=\frac{1}{2}\left(\mathcal{L}\Gamma(f)-2\Gamma(f,\mathcal{L}f)\right)$
show that
\[
\Gamma_{2}(f)=\Gamma_{2}^{L}(f)-\sum_{i=1}^{k}\sum_{j=1}^{k}\left(V_{i}f\right)(V_{i}\sigma_{j})\frac{\partial f}{\partial\xi_{j}}.
\]
By the assumption on $\Gamma_{2}^{L}(f)$ we have
\[
\Gamma_{2}(f)\geqslant \rho\Gamma(f)-\sum_{i=1}^{k}\sum_{j=1}^{k}\left(V_{i}f\right)(V_{i}\sigma_{j})\frac{\partial f}{\partial\xi_{j}}.
\]
Using the Cauchy-Schwarz inequality, the bound on $\sigma$ and the elementary estimate
$ab\leqslant \frac{a^{2}}{2}+\frac{b^{2}}{2}$, we see that
\begin{align*}
\sum_{i,j=1}^{k}\left(V_{i}f\right)(V_{i}\sigma_{j})\frac{\partial f}{\partial\xi_{j}} & \leqslant\left(\sum_{i,j=1}^{k}(V_{i}\sigma_{j})^{2}\right)^{\frac{1}{2}}\left(\sum_{i,j=1}^{k}\left(V_{i}f\right)^{2}\left(\frac{\partial f}{\partial\xi_{j}}\right)^{2}\right)^{\frac{1}{2}}\\
 & \leqslant C_{\sigma}\left(\Gamma(f)\right)^{\frac{1}{2}}\left(\Gamma^{Z}(f)\right)^{\frac{1}{2}}\\
 & \leqslant\frac{C_{\sigma}}{2}\left(\Gamma(f)\right)+\frac{C_{\sigma}}{2}\left(\Gamma^{Z}(f)\right).
\end{align*}
Using this inequality with the previous one give us the desired first curvature-dimension inequality. The second inequality we want to prove is a lower bound on
\[
\Gamma_{2}^{Z}(f):=\frac{1}{2}\left(\mathcal{L}\Gamma^{Z}(f)-2\Gamma^{Z}\left(f,\mathcal{L}f\right)\right),
\]
for which routine computations shows that
\[
\Gamma_{2}^{Z}(f)=\sum_{i,j=1}^{k}\left(V_{i}\frac{\partial f}{\partial\xi_{j}}\right)^{2}\geqslant 0,
\]
as needed.
\end{proof}
In order to prove a gradient bound for the heat semigroup we must make the following assumption
on the existence of a Lyapunov function for the operator $\mathcal{L}$. As in Section \ref{s.RelativDiff}, we say that a smooth function $W:\mathbb{R}^{k}\times\mathbb{R}^{k}\to\mathbb{R}$ is a \emph{Lyapunov function} on $\mathbb{R}^k$ for $\mathcal{L}$ if
\[
\mathcal{L}W\leqslant CW,
\]
for some $C>0$. The existence of a Lyapunov function immediately  implies that $\mathcal{L}$ is the generator of a Markov semigroup $(P_t)_{t\geq 0}$ that uniquely solves the heat equation in $L^\infty$.

Throughout this section, we will need the following assumption.

\begin{assumption}\label{Ref1}
There exists a Lyapunov function $W:\mathbb{R}^{k}\times\mathbb{R}^{k}\to\mathbb{R}$
such that $W\geqslant 1$, $\sqrt{\Gamma (W)}+\sqrt{\Gamma^{Z}(W)}\leqslant CW$,
 for some constant $C>0$ and $\left\{ W\leqslant m\right\} $
is compact for every $m$. Here $\Gamma$ is applied to the first
coordinate of $W$ while $\Gamma^{Z}$ is applied to the second coordinate.
\end{assumption}

We are now ready to prove the main result of this section.

\begin{theorem}[Gradient estimate]\label{GenGrad-Gamma} Suppose Assumption \ref{Ref1} holds and let $P_t$ be the heat semigroup associated to $\mathcal{L}$. If $C_\sigma > 2\rho$ and the operator $L$ satisfies
\[
\Gamma_{2}^{L}(f)\geqslant \rho\Gamma^{L}(f),
\]
then for any $f\in C_{0}^{\infty}\left(\mathbb{R}^{k}\times\mathbb{R}^{k}\right)$, $t \geqslant 0$ and $x \in \mathbb{R}^{k}\times\mathbb{R}^{k}$
\begin{align*}
 & \Gamma\left(P_{t}f\right)(x)+\frac{C_{\sigma}}{C_{\sigma}-2\rho}\Gamma^{Z}\left(P_{t}f\right)(x)\\
 & \leqslant e^{\left(C_{\sigma}-2\rho\right)t}\left(P_{t}\left(\Gamma\left(f\right)\right)\left(x\right)+\frac{C_{\sigma}}{C_{\sigma}-2\rho}P_{t}\left(\Gamma^{Z}\left(f\right)\right)\left(x\right)\right).
\end{align*}

\end{theorem}

\begin{proof}
We fix $t>0$ throughout the proof. For $0<s<t$ and $x=(p,\xi)\in\mathbb{R}^{k}\times\mathbb{R}^{k}$
we denote
\begin{align*}
\varphi_{1}\left(x,s\right) & :=\Gamma\left(P_{t-s}f\right)\left(x\right),\\
\varphi_{2}\left(x,s\right) & :=\Gamma^{Z}\left(P_{t-s}f\right)\left(x\right).
\end{align*}
Then
\begin{align*}
\mathcal{L}\varphi_{1}+\frac{\partial\varphi_{1}}{\partial s} & =2\Gamma_{2}\left(P_{t-s}f\right),\\
\mathcal{L}\varphi_{2}+\frac{\partial\varphi_{2}}{\partial s} & =2\Gamma_{2}^{Z}\left(P_{t-s}f\right).
\end{align*}
Now we would like to find two non-negative smooth functions $a(s)$
and $b(s)$ such that for
\[
\varphi(x,s):=a(s)\varphi_{1}(x,s)+b(s)\varphi_{2}(x,s),
\]
we have
\[
\mathcal{L}\varphi+\frac{\partial\varphi}{\partial s}\geqslant 0.
\]
Then by Theorem \ref{Gen CD inq} we have
\begin{align*}
 & \mathcal{L}\varphi+\frac{\partial\varphi}{\partial s}=\\
 & a^{\prime}(s)\Gamma\left(P_{t-s}f\right)+b^{\prime}(s)\Gamma^{Z}\left(P_{t-s}f\right)+2a(s)\Gamma_{2}\left(P_{t-s}f\right)+2b(s)\Gamma_{2}^{Z}\left(P_{t-s}f\right)\geqslant \\
 & a^{\prime}(s)\Gamma\left(P_{t-s}f\right)+b^{\prime}(s)\Gamma^{Z}\left(P_{t-s}f\right)+2a(s)\left(\left(\rho-\frac{C_{\sigma}}{2}\right)\Gamma\left(P_{t-s}f\right)-\frac{C_{\sigma}}{2}\Gamma^{Z}\left(P_{t-s}f\right)\right)=\\
 & \left(a^{\prime}(s)+a(s)\left(2\rho-C_{\sigma}\right)\right)\Gamma\left(P_{t-s}f\right)+\left(b^{\prime}(s)-a(s)C_{\sigma}\right)\Gamma^{Z}\left(P_{t-s}f\right).
\end{align*}
One can easily see that if
\[
a(s)=e^{\left(C_{\sigma}-2\rho\right)s} \text{ and }b(s)=\frac{C_{\sigma}}{C_{\sigma}-2\rho}e^{\left(C_{\sigma}-2\rho\right)s},
\]
the last expression is $0$. Using the existence of the Lyapunov function
$W$ and a cutoff argument as in  \cite[Theorem 7.3]{BaudoinEMS2014}, we deduce
from a parabolic comparison principle,
\[
P_{t}\left(\varphi\left(\cdot,t\right)\right)\left(x\right)\geqslant\varphi\left(x,0\right).
\]
Observe that
\[
\varphi\left(x,0\right)=a(0)\varphi_{1}\left(x,0\right)+b(0)\varphi_{2}(x,0)=\Gamma\left(P_{t}f\right)(x)+\frac{C_{\sigma}}{C_{\sigma}-2\rho}\Gamma^{Z}\left(P_{t}f\right)(x),
\]
while
\begin{align*}
P_{t}\left(\varphi\left(\cdot,t\right)\right)\left(x\right) & =a(t)P_{t}\left(\Gamma\left(f\right)\right)\left(x\right)+b(t)P_{t}\left(\Gamma^{Z}\left(f\right)\right)\left(x\right)\\
 & =e^{\left(C_{\sigma}-2\rho\right)t}\left(P_{t}\left(\Gamma\left(f\right)\right)\left(x\right)+\frac{C_{\sigma}}{C_{\sigma}-2\rho}P_{t}\left(\Gamma^{Z}\left(f\right)\right)\left(x\right)\right).
\end{align*}

\end{proof}

\begin{corollary}[Poincar\'e type inequality] If $C_\sigma > 2\rho$ then
for any $f\in C_{0}^{\infty}\left(\mathbb{R}^{k}\times\mathbb{R}^{k}\right)$ and $t \geqslant 0$
\[
P_{t}\left(f^{2}\right)-\left(P_{t}f\right)^{2}\leqslant 2\frac{e^{\left(C_{\sigma}-2\rho\right)t}-1}{C_{\sigma}-2\rho}\left(P_{t}\left(\Gamma\left(f\right)\right)+\frac{C_{\sigma}}{C_{\sigma}-2\rho}P_{t}\left(\Gamma^{Z}\left(f\right)\right)\right).
\]

\end{corollary}

\begin{proof}
Since $\Gamma^{Z}\left(f\right):=\left\Vert \nabla_{\xi}f\right\Vert ^{2}\geqslant 0$
and $P_{t}\left(f^{2}\right)-\left(P_{t}f\right)^{2}=2\int_{0}^{t}P_{s}\left(\Gamma\left(P_{t-s}f\right)\right)ds$,
then
\begin{align*}
 & \int_{0}^{t}P_{s}\left(\Gamma\left(P_{t-s}f\right)+\frac{C_{\sigma}}{C_{\sigma}-2\rho}\Gamma^{Z}\left(P_{t-s}f\right)\right)ds\geqslant\\
 & \frac{1}{2}\int_{0}^{t}2P_{s}\left(\Gamma\left(P_{t-s}f\right)\right)ds=\frac{1}{2}\left(P_{t}\left(f^{2}\right)-\left(P_{t}f\right)^{2}\right).
\end{align*}
By Theorem \ref{GenGrad-Gamma} we have that
\begin{align*}
 & \int_{0}^{t}P_{s}\left(\Gamma\left(P_{t-s}f\right)+\frac{C_{\sigma}}{C_{\sigma}-2\rho}\Gamma^{Z}\left(P_{t-s}f\right)\right)ds\leqslant\\
 & \int_{0}^{t}e^{\left(-2\rho+C_{\sigma}\right)(t-s)}P_{s}\left(P_{t-s}\left(\Gamma\left(f\right)\right)+\frac{C_{\sigma}}{C_{\sigma}-2\rho}P_{t-s}\left(\Gamma^{Z}\left(f\right)\right)\right)ds=\\
 & \left(P_{t}\left(\Gamma\left(f\right)\right)+\frac{C_{\sigma}}{C_{\sigma}-2\rho}P_{t}\left(\Gamma^{Z}\left(f\right)\right)\right)\int_{0}^{t}e^{\left(C_{\sigma}-2\rho\right)(t-s)}ds=\\
 & \frac{e^{\left(C_{\sigma}-2\rho\right)t}-1}{C_{\sigma}-2\rho}\left(P_{t}\left(\Gamma\left(f\right)\right)+\frac{C_{\sigma}}{C_{\sigma}-2\rho}P_{t}\left(\Gamma^{Z}\left(f\right)\right)\right).
\end{align*}
So that
\[
P_{t}\left(f^{2}\right)-\left(P_{t}f\right)^{2}\leqslant 2\frac{e^{\left(C_{\sigma}-2\rho\right)t}-1}{C_{\sigma}-2\rho}\left(P_{t}\left(\Gamma\left(f\right)\right)+\frac{C_{\sigma}}{C_{\sigma}-2\rho}P_{t}\left(\Gamma^{Z}\left(f\right)\right)\right).
\]

\end{proof}

\subsection{Application to Riemannian manifolds}\label{5.1.1}
To illustrate the results in Section \ref{5.1} we study a large  class of examples based on Brownian motion in a Riemannian manifold. For general background on general theory of stochastic analysis on manifolds we refer to \cite{BismutBook1984, ElworthyBook, HsuEltonBook}. Consider a complete Riemannian manifold $\left(M,g\right)$ of dimension $d$ which is isometrically embedded in $\mathbb{R}^{k}$ for some $k$. Let $B_{t}$ be a Brownian motion on $M$ and consider the process $\mathbf{X}_{t}=\left(B_{t},\int_{0}^{t}\sigma\left(B_{s}\right)ds\right)$ where $\sigma:M\to\mathbb{R}^k$ satisfies \eqref{C-sigma1} and
\[
\left|\sigma(p)-\sigma(\tilde{p})\right|\leqslant C_{\sigma}d_{M}\left(p,\tilde{p}\right),
\]
for all $p,\tilde{p}\in M$ where $d_{M}$ is the intrinsic Riemannian
distance on $M$.
We can write the generator of $B_{t}$ as
\[
\Delta_{p}=\sum_{i=1}^{k}P_{i}^{2},
\]
for some vector fields $P_{i}$ on $\mathbb{R}^{k}$ (see for instance \cite[Theorem 3.1.4]{HsuEltonBook}). The generator of $X_{t}$ is
\[
\mathcal{L}=\Delta_{p}+\sum_{i=1}^{k}\sigma_{i}\left(p\right)\frac{\partial}{\partial\xi_{i}},
\]
for functions $f\left(p,\xi\right)\in M\times\mathbb{R}^{k}$ where
$p\in M,\xi\in\mathbb{R}^{k}$.

To apply Theorem \ref{GenGrad-Gamma} we first need to construct an appropriate Lyapunov function $W$ for the operator $\mathcal{L}$ satisfying Assumption \ref{Ref1}. Once we construct $W$, we will spend the rest of the section verifying Assumption \ref{Ref1} for $W$.  For this, we  assume that the Ricci curvature $\operatorname{Ric} \geqslant \rho$  for some $\rho \in \mathbb{R}$.  Then it is known from the Li-Yau upper and lower bounds in \cite{LiYau1986} that the heat kernel $p(x,y,t)$ of $M$ satisfies the following Gaussian estimates. Namely, for some $\tau >0$
 \begin{align*}
& \frac{c_1}{Vol(B(p_0,\sqrt
\tau))} \exp
\left(-\frac{ c_2d_M(p_0,p_1)^2}{\tau}\right) \leqslant p(p_0,p_1,\tau)
\\
& \leqslant \frac{c_3}{Vol(B(p_0,\sqrt
\tau))} \exp
\left(-\frac{c_4d_M(p_0,p_1)^2}{\tau}\right),
\end{align*}
where $d_{M}$ is the Riemannian distance in $M$ and $p_0, p_1 \in M$. Consider now the smooth  Lyapunov function
\begin{equation}\label{LyapunovManifold}
W\left(p,\xi\right):=K+\left\Vert \xi\right\Vert ^{2}-\ln p(p_0,p,\tau), p\in M, \xi \in\mathbb{R}^{k},
\end{equation}
where $p_{0}$ is an arbitrary fixed point in $M$, and $K$ is a constant large enough so that $W \geqslant 1$.

\begin{lemma}\label{lem4.5}
 The function $W$ defined in \eqref{LyapunovManifold} is smooth and satisfies the following properties,
\begin{align*}
& W\geqslant 1,
\\
& \left\Vert \nabla_{\xi}W\right\Vert +\left\Vert \nabla_{p}W\right\Vert \leqslant CW,
\\
& \mathcal{L}W\leqslant CW \text{ for some constant } C>0,
\\
& \left\{ W\leqslant m\right\}  \text{ is compact for every } m.
\end{align*}
Here $\nabla_{p}$ is the Riemannian gradient on $M$ and $\nabla_{\xi}$ is the Euclidean gradient on $\mathbb{R}^{k}$.
\end{lemma}

\begin{proof}
From estimates for logarithmic derivatives of the heat kernel in \cite{Hamilton1993b, LiYau1986}, one has for some constants $C_1,C_2>0$
\begin{align}
\left\Vert \nabla_{p}\ln p\left(p_{0},p,\tau\right)\right\Vert ^{2} & \leqslant C_{1}+C_{2}d_{M}\left(p_{0},p\right)^{2},\label{Li-Yau-Ineq1}\\
\Delta_{p}\left(-\ln p\left(p_{0},p,\tau\right)\right) & \leqslant C_{1}+C_{2}d_{M}\left(p_{0},p\right)^{2}.\label{Li-Yau-Ineq2}
\end{align}
We can then conclude with the Li-Yau upper and lower Gaussian bounds. To see this note that the Gaussian bounds can be rearranged as

\begin{align}
d_{M}\left(p_{0},p\right)^{2} & \leqslant-\frac{\tau}{c_{4}} \ln\left(\frac{\operatorname{Vol}(B\left(p_{0},\sqrt{\tau}\right))}{c_{3}}p\left(p_{0},p,\tau\right)\right)
\notag
\\
 & \leqslant C\left(K-\ln\left(p\left(p_{0}, p, \tau\right)\right)\right)
 \label{Lya-proof1}
\end{align}
for a fixed $\tau \geqslant 0$ and a constant $C>0$. Hence, $\left\Vert \nabla_{\xi}W\right\Vert +\left\Vert \nabla_{p}W\right\Vert \leqslant CW$ can be shown using \eqref{Li-Yau-Ineq1}, \eqref{Lya-proof1} and the inequality $\left(1+x\right)^{\frac{1}{2}}\leqslant1+cx$ for $x\geqslant 0$ and $c\geqslant\frac{1}{2}$. On the other hand, $\mathcal{L}W\leqslant CW$ can be shown using \eqref{Li-Yau-Ineq2}, \eqref{Lya-proof1}, the Cauchy-Schwarz inequality, and the Lipschitz property of $\sigma$. Finally, the fact that $\left\{ W\leqslant m\right\}$ is compact for every $m$ also follows from the Li-Yau upper and lower Gaussian bounds.

\end{proof}

Lemma \ref{lem4.5} proves that $W$ defined by \eqref{LyapunovManifold} is a Lyapunov function satisfying Assumption \ref{Ref1}. As a consequence, Theorem \ref{GenGrad-Gamma} can be applied to complete Riemannian manifolds with $\operatorname{Ric}\geqslant\rho$ since the condition $\operatorname{Ric}\geqslant\rho$ is equivalent to
\[
\Gamma_{2}^{\Delta}(f)\geqslant \rho\Gamma\left(f\right).
\]

\subsection{Coupling}\label{5.2}

Let $(M,g)$ be a complete connected $d-$dimensional Riemannian manifold. In this section $M$ is not necessarily embedded in $\mathbb{R}^{k}$. We assume the existence of a map $\sigma: M \to \mathbb{R}^{k}$ for some $k\geqslant 1$
that is globally $C_\sigma$-Lipschitz map in the sense that
\begin{equation}\label{Lip1}
\left|\sigma\left(p\right)-\sigma\left(\widetilde{p}\right)\right|\leqslant C_{\sigma}d_{M}\left(p,\widetilde{p}\right),
\end{equation}
for all $p,\widetilde{p}\in M$. Here we denote by $d_M$ the Riemannian distance on $M$, and by $d_E$ we denote the Euclidean metric in $\mathbb{R}^k$. We will consider the joint process
\begin{equation}\label{hypo}
\mathbf{X}_{t}=\left(B_{t},\int_{0}^{t}\sigma\left(B_{s}\right)ds\right),
\end{equation}
on the product space $M\times \mathbb{R}^k$ where $B_{t}$ is Brownian motion on $M$.

Let $P_t$ be the associated heat semigroup. We consider functions on $M \times \mathbb{R}^{k}$ with $f \left( p, \xi\right), p \in M, \xi \in \mathbb{R}^{k}$. Recall that the operators $\nabla_p$ and $\Delta_p$ act on the variable $p$ for  $f\left( p, \xi \right)$, where $\Delta_p$ is the Laplace-Beltrami operator. We use $\nabla_\xi$ for the usual Euclidean gradient. Given a Riemannian metric $g$, for all $p\in M$ and $v\in T_p M$ we denote $\left\Vert v\right\Vert =g_{p}\left(v,v\right)^{\frac{1}{2}}$.  Our main result of this section is a bound on $\left\Vert \nabla_{p}P_{t}f\right\Vert$ for functions $f\in C^{\infty}\left(M\times\mathbb{R}^{k}\right)$ with bounded Hessian. This will be a generalization of the result obtained in Section \ref{EuclidCoupling}.

Let us recall the notion of a coupling of diffusions on a manifold $M$. Suppose $X_{t}$ and $\widetilde{X}_{t}$ are $M$-valued diffusions starting at $x,\widetilde{x}\in M$ on a probability space $\left(\Omega, \mathcal{F}, \mathbb{P} \right)$. Then by a coupling of $X_{t}$ and $\widetilde{X}_{t}$ we call a $C\left( \mathbb{R}_{+}, M \times M \right)$-valued random variable  $\left( X_{t}, \widetilde{X}_{t} \right)$  on the probability space $\left(\Omega, \mathcal{F}, \mathbb{P} \right)$ such that the marginal processes for $\left( X_{t},\widetilde{X}_{t} \right)$ have the same laws as $X_{t}$ and $\widetilde{X}_{t}$. Let $\mathbb{P}^{\left(x,\widetilde{x}\right)}$ be the distribution of $\left( X_t,\widetilde{X}_t \right)$, so that  $\mathbb{P}^{(x,\widetilde{x})}\left( X_{0}=x,\widetilde{X}_{0}=\widetilde{x}\right)=1$. We denote by
$\mathbb{E}^{\left(x,\widetilde{x}\right)}$ the expectation with respect to the probability measure $\mathbb{P}^{\left(x,\widetilde{x}\right)}$.


In \cite{vonRenesseSturm2005, WangFY1997b, vonRenesse2004a} it has been shown that if we assume $\operatorname{Ric}\left(M\right)\geqslant K$ for some $K \in \mathbb{R}$, then there exists a Markovian coupling of Brownian motions $\left(B_{t}\right)_{t\geqslant 0}$ and $\left(\widetilde{B}_{t}\right)_{t \geqslant 0}$ on $M$ starting at $p$ and $\widetilde{p}$  such that
\begin{equation}
d_{M}\left(B_{t},\widetilde{B}_{t}\right)\leqslant e^{-Kt/2}d_{M}\left(p,\widetilde{p}\right)\label{eq:Sturm1}
\end{equation}
for all $t\geqslant0$, $\mathbb{P}^{\left(p,\widetilde{p}\right)}$-almost surely. This construction is known as a \emph{coupling by parallel transport}. This coupling can be constructed using stochastic differential equations as in \cite{WangFY1997b, Cranston1991}, or by a central limit theorem argument for the geodesic random walks as in \cite{vonRenesse2004a}.  It turns out that the existence of the coupling satisfying \eqref{eq:Sturm1} is equivalent to
\begin{equation}\label{RiemGrad}
\left\Vert \nabla P_{t}f\right\Vert \leqslant e^{-Kt}P_{t}\left(\left\Vert \nabla f\right\Vert \right),
\end{equation}
for all $f\in C_{0}^{\infty}(M)$ and all $t>0$. We also point out that in  \cite{PascuPopescu2016,Pascu-Popescu2018}, M.~Pascu and I.~Popescu constructed explicit Markovian couplings where equality in  \eqref{eq:Sturm1} is attained for $t\geqslant 0$ dependent on $K$ and given some extra geometric assumptions.

The coupling by parallel transport that gives \eqref{eq:Sturm1} is in the elliptic setting. In this section, we will use the coupling by parallel transport to induce a coupling for \eqref{hypo} in the hypoelliptic setting. We will then use this coupling to prove gradient bounds for $(P_t)_{t\geq 0}$. Before stating the result on the gradient bound, we have the following
proposition.

\begin{proposition}\label{Prop:Limsup}
Let $\left(M,g\right)$ be a Riemannian manifold. If $f\in C^{1}\left(M\right)$
then
\begin{equation}\label{eq:Limsup}
\lim_{r\to0}\sup_{\widetilde{p}:0<d_{M}\left(p,\widetilde{p}\right)\leqslant r}\frac{\left|f\left(p\right)-f\left(\widetilde{p}\right)\right|}{d_{M}\left(p,\widetilde{p}\right)}=\left\Vert \nabla f(p)\right\Vert.
\end{equation}

\end{proposition}

\begin{proof}
Let $p,\widetilde{p}\in M$ with $T=d_{M}\left(p,\widetilde{p}\right)$ and
consider a unit speed geodesic $\gamma:\left[0, T\right]\to M$
such that $\gamma\left(0\right)=\widetilde{p}$ and $\gamma\left(T\right)=p$. Then

\begin{align*}
\left|f\left(p\right)-f\left(\widetilde{p}\right)\right| & =\left|\int_{0}^{d\left(p,\widetilde{p}\right)}g\left(\nabla f\left(\gamma\left(s\right)\right),\gamma^{\prime}(s)\right)ds\right|\\
 & \leqslant\int_{0}^{d\left(p,\widetilde{p}\right)}\left|g\left(\nabla f\left(\gamma\left(s\right)\right),\gamma^{\prime}(s)\right)\right|ds\\
 & \leqslant\max_{0\leqslant s\leqslant d(p,\widetilde{p})}\left\Vert \nabla f\left(\gamma(s)\right)\right\Vert \cdot d\left(p,\widetilde{p}\right)
\end{align*}
where we used the Cauchy-Schwarz inequality. Since $p,\widetilde{p}$ are arbitrary,
dividing out both sides by $d\left(p,\widetilde{p}\right)$ we have that
\[
\lim_{r\to0}\sup_{\widetilde{p}:0<d_{M}\left(p,\widetilde{p}\right)\leqslant r}\frac{\left|f\left(p\right)-f\left(\widetilde{p}\right)\right|}{d_{M}\left(p,\widetilde{p}\right)}\leqslant\left\Vert \nabla f(p)\right\Vert .
\]

On the other hand, find a unit speed geodesic $\gamma:\left(-\epsilon, \epsilon\right) \to M$
such that $\gamma (0)=p$ and  $\gamma^{\prime}\left(0\right)=\nabla f(p)/\left\Vert \nabla f(p)\right\Vert$. Define $F\left(s\right)=f\left(\gamma\left(s\right)\right)$. Since
$F^{\prime}\left(s\right)=g\left(\nabla f\left(\gamma\left(s\right)\right),\gamma^{\prime}(s)\right)$,
then
\[
F^{\prime}\left(0\right)=g\left(\nabla f\left(p\right),\frac{\nabla f(p)}{\left\Vert \nabla f(p)\right\Vert }\right)=\left\Vert \nabla f(p)\right\Vert .
\]
Now by the definition of the derivative we have that
\[
\lim_{h\to0}\frac{F(h)-F(0)}{h}\to\left\Vert \nabla f(p)\right\Vert, \]
which means we have that the left hand side of \eqref{eq:Limsup}
must be at least $\left\Vert \nabla f(p)\right\Vert $. This proves \eqref{eq:Limsup}.

\end{proof}

The following lemma gives an estimate for $\left|f(p,\xi)-f(\widetilde{p}, \widetilde{\xi})\right|$ on $M\times \mathbb{R}^k$.

\begin{lemma}
\label{prop:Taylor} Let $\left(M, g\right)$ be a complete Riemannian
manifold which is assumed to be embedded in $\mathbb{R}^{k}$.  For a function $f\left( p, \xi\right)$ we denote by $\nabla_{p}f$  the Riemannian gradient acting on $p$, and by $\nabla_{\xi}f$  the Euclidean gradient acting on $\xi$. If $f\in C^{2} \left( M\times\mathbb{R}^{k} \right)$ with a bounded Hessian, then there exists a $C_{f}>0$ depending on a bound on the Hessian of $f$ such that
\begin{align*}
\left|f(p,\xi)-f(\widetilde{p},\widetilde{\xi})\right| & \leqslant\left\Vert \nabla_{p}f\left(\widetilde{p},\widetilde{\xi}\right)\right\Vert d_{M}\left(p,\widetilde{p}\right)+\left\Vert \nabla_{\xi}f\left(\widetilde{p},\widetilde{\xi}\right)\right\Vert d_{E}\left(\xi,\widetilde{\xi}\right)\\
 & +C_{f}\left(d_{M}\left(p,\widetilde{p}\right)+d_{E}\left(\xi,\widetilde{\xi}\right)\right)^{2}
\end{align*}
for any $\left(p,\xi\right),\left(\widetilde{p},\widetilde{\xi}\right)\in M\times\mathbb{R}^{k}$. \end{lemma}
\begin{proof}
Let $p,\widetilde{p}\in M$ with $T_{1}=d_{M}\left(p,\widetilde{p}\right)$
and consider a unit speed geodesic $\gamma:\left[0,T_1\right]\to M$
such that $\gamma\left(0\right)=\widetilde{p}$ and $\gamma\left(T_1\right)=p$. Let $\xi,\widetilde{\xi}\in\mathbb{R}^{k}$ with $T_{2}=d_{E}\left(\xi,\widetilde{\xi}\right)$
and consider $\beta(s)=\frac{s}{d_{E}\left(\xi,\widetilde{\xi}\right)}\left(\xi-\widetilde{\xi}\right)+\widetilde{\xi}$
on $-\infty\leqslant s\leqslant T_{2}$ such that $\beta(0)=\widetilde{\xi}$ and
$\beta(T_{2})=\xi$. Extend $\gamma$ to $\left[-\epsilon,T_1\right]$
for some $\epsilon>0$ and define $F\left(t,s\right)=f\left(\gamma\left(t\right),\beta(s)\right)$.
By an estimate on the remainder of Taylor's approximation there exists a $C_{f}>0$ depending
only on a bound on the Hessian of $f$ such that
\[
\left|F(t,s)-F\left(0,0\right)\right|  \leqslant\left|F_{t}(0,0)t+F_{s}(0,0)s\right|+C_{f}\left(t+s\right)^{2}.
\]
By the chain rule we have
\begin{align*}
F_t(0,0)=\frac{d}{dt}\left[f\left(\gamma\left(t\right),\beta(0)\right)\right]_{t=0} & =\left\langle \nabla_{p}f\left(\gamma\left(0\right),\beta(0)\right),\gamma^{\prime}\left(0\right)\right\rangle \\
 & \leqslant\left\Vert \nabla_{p}f\left(\gamma\left(0\right),\beta(0)\right)\right\Vert =\left\Vert \nabla_{p}f\left(\widetilde{p},\widetilde{\xi}\right)\right\Vert .
\end{align*}
 Similarly $F_s(0,0)=\frac{d}{ds}\left[f\left(\gamma\left(0\right),\beta(s)\right)\right]_{s=0}\leqslant\left\Vert \nabla_{\xi}f\left(\widetilde{p},\widetilde{\xi}\right)\right\Vert$. Then
\begin{align*}
 & \left|f(p,\xi)-f(\widetilde{p},\widetilde{\xi})\right|=\left|F(T_1,T_2)-F\left(0,0\right)\right|\\
 & \leqslant\left\Vert \nabla_{p}f\left(\widetilde{p},\widetilde{\xi}\right)\right\Vert T_{1}+\left\Vert \nabla_{\xi}f\left(\widetilde{p},\widetilde{\xi}\right)\right\Vert T_{2}+C_{f}\left(T_{1}+T_{2}\right)^{2},
\end{align*}
as needed.
\end{proof}

We are now ready to state and prove the main theorem of this section. We start by considering the coupling of Brownian motions $\left(B_{t}, \widetilde{B}_{t}\right)$ starting at $\left( p,\widetilde{p} \right)$ by parallel transport satisfying \eqref{eq:Sturm1}, as introduced in \cite{vonRenesseSturm2005,vonRenesse2004a} . This coupling induces a coupling $\mathbb{P}^{\left(\mathbf{x},\widetilde{\mathbf{x}}\right)}$ on $\left(M\times\mathbb{R}^{d}\right)\times\left(M\times\mathbb{R}^{d}\right)$
for two Kolmogorov type diffusions
\[
\mathbf{X}_{t}=\left(B_{t},\xi+\int_{0}^{t}\sigma\left(B_{s}\right)ds\right)\text{ and }\widetilde{\mathbf{X}}_{t}=\left(\widetilde{B}_{t},\xi+\int_{0}^{t}\sigma\left(\widetilde{B}_{s}\right)ds\right),
\]
started at $\mathbf{x}=\left(p,\xi\right)$ and $\widetilde{\mathbf{x}}=\left(\widetilde{p}, \xi\right)$ respectively.

\begin{theorem}[Bakry-\'Emery type estimate]\label{GenGradient}
Let $M$ be a complete connected Riemannian manifold such that  $\text{Ric}\left(M\right)\geqslant K$ for some
$K\in\mathbb{R}$. Let $\sigma$ be a $C_\sigma-$Lipschitz
map as in \eqref{Lip1} and $f\in C^{2}\left(M\times\mathbb{R}^{k}\right)$ with a bounded Hessian.
Then for every $q\geqslant 1$ and $t \geqslant 0$,
\[
\left\Vert \nabla_{p}P_{t}f\right\Vert ^{q}\leqslant P_{t}\left(\left(K_{1}(t)\left\Vert \nabla_{p}f\right\Vert +K_{2}(t)\left\Vert \nabla_{\xi}f\right\Vert \right)^{q}\right),
\]
where
\[
K_{1}(t)=e^{-Kt/2} \text{ and } K_{2}(t)=\begin{cases}
C_{\sigma}t & K=0\\
C_{\sigma}\frac{1-e^{-Kt/2}}{K/2} & K\neq0.
\end{cases}
\]

\end{theorem}

\begin{proof}
As before let $d_{M}$ be the Riemannian distance on $M$, and let $d_{E}$ be the Euclidean
distance on $\mathbb{R}^{k}$. Take $\mathbf{x}=\left(p,\xi\right)\in M\times\mathbb{R}^{k}$ and
$\widetilde{\mathbf{x}}=\left(\widetilde{p},\xi\right)\in M\times\mathbb{R}^{k}$. If $K\neq 0$, we consider the coupling by parallel transport of Brownian motions $\left(B_{t},\widetilde{B}_{t}\right)$ starting at $\left(p,\widetilde{p}\right)$.
This coupling gives us that
\begin{equation}
d_{M}\left(B_{t},\widetilde{B}_{t}\right)\leqslant e^{-Kt/2}d_{M}\left(p,\widetilde{p}\right),\label{Synchronous1}
\end{equation}
for all $t\geqslant0$. Denote $Y_{t}=\xi+\int_{0}^{t}\sigma(B_{s})ds$ and $\widetilde{Y}_{s}=\xi+\int_{0}^{t}\sigma(\widetilde{B}_{s})ds$.
If $K\neq 0$ then
\begin{align}
d_{E}\left(Y_{t},\widetilde{Y}_{t}\right)\leqslant\int_{0}^{t}\left|\sigma\left(B_{s}\right)-\sigma\left(\widetilde{B}_{s}\right)\right|ds\leqslant C_{\sigma}\int_{0}^{t}d_{M}\left(B_{s},\widetilde{B}_{s}\right)ds\nonumber \\
\leqslant C_{\sigma}d_{M}\left(p,\widetilde{p}\right)\int_{0}^{t}e^{-Ks/2}ds=C_{\sigma}\left(\frac{1-e^{-Kt/2}}{K/2}\right)d_{M}\left(p,\widetilde{p}\right),\label{eq:E-Estimate1}
\end{align}
where we used \eqref{Lip1} and \eqref{eq:Sturm1} . If $K=0$,
we consider the same coupling for the Brownian motions $\left(B_{t},\widetilde{B}_{t}\right)$ starting at $\left(p,\widetilde{p}\right)$ so that
\begin{equation}
d_{M}\left(B_{t},\widetilde{B}_{t}\right)\leqslant d_{M}\left(p,\widetilde{p}\right),\label{Synchronous2}
\end{equation}
for all $t\geqslant 0$. A similar computation as in \eqref{eq:E-Estimate1} gets us the estimate
\begin{equation}
d_{E}\left(Y_{t},\widetilde{Y}_{t}\right)\leqslant C_{\sigma}td_{M}\left(p,\widetilde{p}\right),\label{E-Estimate2}
\end{equation}
 from \eqref{Synchronous2}. Combining \eqref{Synchronous1} and \eqref{Synchronous2} we get
\begin{equation}
d_{M}\left(B_{t},\widetilde{B}_{t}\right)\leqslant K_{1}(t)d_{M}\left(p,\widetilde{p}\right),\label{Synchronous3}
\end{equation}
while combining \eqref{eq:E-Estimate1} and \eqref{E-Estimate2} we have
\begin{equation}
d_{E}\left(Y_{t},\widetilde{Y}_{t}\right)\leqslant K_{2}(t)d_{M}\left(p,\widetilde{p}\right),\label{eq:GenEsti2}
\end{equation}
for all $t\geqslant0$, where all of these inequalities hold $\mathbb{P}^{\left(\bf{x},\widetilde{\bf{x}}\right)}-$almost surely. By Lemma $\ref{prop:Taylor}$, there exists a $C_{f}\geqslant 1$ depending on a bound on the Hessian of $f\in C^{2}\left(M\times\mathbb{R}^{k}\right)$ such that
\begin{align}
\left|f\left(B_t,Y_t\right)-f\left(\widetilde{B_t},\widetilde{Y_t}\right)\right| & \leqslant\left\Vert \nabla_{p}f\left(\widetilde{B_t},\widetilde{Y_t}\right)\right\Vert d_{M}\left(B_t,\widetilde{B_t}\right)+\left\Vert \nabla_{\xi}f\left(\widetilde{B_t},\widetilde{Y_t}\right)\right\Vert d_{E}\left(Y_t,\widetilde{Y_t}\right)\nonumber \\
 & +C_{f}\left(d_{M}\left(B_t,\widetilde{B_t}\right)+d_{E}\left(Y_t,\widetilde{Y_t}\right)\right)^{2},\label{eq:GenTaylor1}
\end{align}
for all $t\geqslant 0 $, $\mathbb{P}^{\left(\bf{x},\widetilde{\bf{x}}\right)}-$almost surely.

Using inequalities \eqref{Synchronous3}, \eqref{eq:GenEsti2}
and \eqref{eq:GenTaylor1}, we have that for $f \in C^{2}\left(M\times\mathbb{R}^{k}\right)$

\begin{align*}
 & \left|P_{t}f\left(p,\xi\right)-P_{t}f\left(\widetilde{p},\xi\right)\right|=\left|\mathbb{E}^{\left(\mathbf{x},\widetilde{\mathbf{x}}\right)}\left[f\left(B_{t},Y_{t}\right)-f\left(\widetilde{B}_{t},\widetilde{Y}_{t}\right)\right]\right|\\
 & \leqslant\mathbb{E}^{\left(\mathbf{x},\widetilde{\mathbf{x}}\right)}\left[\left\Vert \nabla_{p}f\left(\widetilde{B}_{t},\widetilde{Y}_{t}\right)\right\Vert d_{M}\left(B_{t},\widetilde{B}_{t}\right)+\left\Vert \nabla_{\xi}f\left(\widetilde{B}_{t},\widetilde{Y}_{t}\right)\right\Vert d_{E}\left(Y_{t},\widetilde{Y}_{t}\right)\right]\\
 & +C_{f}\mathbb{E}^{\left(\mathbf{x},\widetilde{\mathbf{x}}\right)}\left[d_{M}\left(B_{t},\widetilde{B}_{t}\right)+d_{E}\left(Y_{t},\widetilde{Y}_{t}\right)\right]^{2}\\
 & \leqslant\mathbb{E}^{\left(\mathbf{x},\widetilde{\mathbf{x}}\right)}\left[K_{1}(t)\left\Vert \nabla_{p}f\left(\widetilde{B}_{t},\widetilde{Y}_{t}\right)\right\Vert +K_{2}(t)\left\Vert \nabla_{\xi}f\left(\widetilde{B}_{t},\widetilde{Y}_{t}\right)\right\Vert \right]d_{M}\left(p,\widetilde{p}\right)\\
 & +C_{f}\left(K_1(t)+K_2(t)\right)^{2}d_{M}\left(p,\widetilde{p}\right)^{2}.
\end{align*}
Using Jensen's inequality for $q\geqslant1$ we have
\begin{align*}
 & \left|P_{t}f\left(p,\xi\right)-P_{t}f\left(\widetilde{p},\xi\right)\right|\\
 & \leqslant\left(\mathbb{E}^{\left(\mathbf{x},\widetilde{\mathbf{x}}\right)}\left[\left(K_{1}(t)\left\Vert \nabla_{p}f\left(\widetilde{B}_{t},\widetilde{Y}_{t}\right)\right\Vert +K_{2}(t)\left\Vert \nabla_{\xi}f\left(\widetilde{B}_{t},\widetilde{Y}_{t}\right)\right\Vert \right)^{q}\right]\right)^{\frac{1}{q}}d_{M}\left(p,\widetilde{p}\right)\\
 & +C_{f}\left(K_1(t)+K_2(t)\right)^{2}d_{M}\left(p,\widetilde{p}\right)^{2}.
\end{align*}
 Dividing the last inequality out by $d_{M}\left(p,\widetilde{p}\right)$
we have that
\begin{align*}
\frac{\left|P_{t}f\left(p,\xi\right)-P_{t}f\left(\widetilde{p},\xi\right)\right|}{d_{M}\left(p,\widetilde{p}\right)} & \leqslant\left[P_{t}\left(\left(K_{1}(t)\left\Vert \nabla_{p}f\right\Vert +K_{2}(t)\left\Vert \nabla_{\xi}f\right\Vert \right)^{q}\right)\left(\widetilde{p},\xi\right)\right]^{\frac{1}{q}}\\
 & +C_{f}\left(K_1(t)+K_2(t)\right)^{2}d_{M}\left(p,\widetilde{p}\right).
\end{align*}
 Since
\[
\lim_{r\to0}\sup_{\widetilde{p}:0<d_{M}\left(p,\widetilde{p}\right)\leqslant r}\frac{\left|P_{t}f\left(p,\xi\right)-P_{t}f\left(\widetilde{p},\xi\right)\right|}{d_{M}\left(p,\widetilde{p}\right)}=\left\Vert \nabla_{p}P_{t}f\left(p,\xi\right)\right\Vert
\]
by Proposition \ref{Prop:Limsup}, we have the desired result.
\end{proof}

\begin{remark}
The constants obtained in Theorem \ref{GenGradient} using the coupling technique are sharper than the constants in Theorem \ref{GenGrad-Gamma} using $\Gamma$-calculus. The trade off here being that the $\Gamma$-calculus approach allows for the result to be proven for a wider class of Kolmogorov type diffusions.
\end{remark}

\begin{remark}
We note that when applying the triangle inequality to the right hand sides of the inequalities in Propositions \ref{Kolmogorov1},  \ref{EuclideanCouplingBakryEmery}, we recover Theorem \ref{GenGradient} when the manifold $M=\mathbb{R}^d$. Here we have $k=d$, $\sigma(\bf{x})= \bf{x}$ and $C_\sigma=1$.

\end{remark}

\begin{example}[Velocity spherical Brownian motion]
The velocity spherical Brownian is a diffusion process which takes values in $T^1\mathcal{M}$, the unit tangent bundle of a Riemannian manifold of finite volume. The generator is of the form
$$L=\frac{\sigma^2}{2}\Delta_v+\kappa\xi.$$
It was introduced in \cite{AngstBailleulTardif2015} and further studied in \cite{BaudoinTardif2018}. When $\mathcal{M}=\mathbb{R}^{d+1}$ and $\sigma=\kappa=1$ the diffusion is of the form $\mathbf{X}_{t}=(B_t,\int_0^tB_s ds)$ where $B_t$ is a Brownian motion on the $d$-dimensional sphere $\mathbb{S}^d$. Here we take $\mathbb{S}^d$ to have the usual embedding in $\mathbb{R}^{d+1}$, that is, $\mathbb{S}^{d}=\left\{ \mathbf{x}\in\mathbb{R}^{d+1}\mid\left|\mathbf{x}\right|=1\right\} $.
Let $d_{\mathbb{S}^{d}}$ be the spherical distance and $d_{E}\left(\mathbf{x},\mathbf{y}\right)=\left|\mathbf{x}-\mathbf{y}\right|$
is the Euclidean distance in $\mathbb{R}^{d+1}$. The explicit spherical
distance is given by
\[
d_{\mathbb{S}^{d}}\left(\mathbf{x},\mathbf{y}\right)=\cos^{-1}\left(\mathbf{x}\cdot\mathbf{y}\right),
\]
for $\mathbf{x},\mathbf{y}\in\mathbb{S}^{d}$, where the standard Euclidean inner product is used. It is easy to see that
\begin{equation}
d_{E}\left(\mathbf{x},\mathbf{y}\right)\leqslant d_{\mathbb{S}^{d}}\left(\mathbf{x},\mathbf{y}\right),\label{EucLeqSphr}
\end{equation}
for all $\mathbf{x},\mathbf{y}\in\mathbb{S}^{d}$ since the Riemannian structure of $\mathbb{S}^d$ is induced by the Euclidean structure of the ambient space $\mathbb{R}^{d+1}$.
Inequality \eqref{EucLeqSphr} shows that $\sigma:\mathbb{S}^d\to \mathbb{R}^{d+1}$ is a $C_{\sigma}=1$-Lipschtiz map. Thus we can apply Theorem \ref{GenGradient} to the manifold $M=\mathbb{S}^d$,  since $Ric=(d-1)g$ where $g$ is the Riemannian metric.

\end{example}

\begin{example}
Let $k=1$ and fix a $p_{0}\in M$. We consider the map $\sigma:M\to\mathbb{R}$
defined by
\[
\sigma(p)=d_{M}(p,p_{0}).
\]
Note that this map is globally $1$-Lipschitz since
\[
\left|\sigma(p)-\sigma(\tilde{p})\right| =\left|d_{M}(p,p_{0}))-d_{M}(\tilde{p},p_{0})\right|  \leqslant d_{M}\left(p,\tilde{p}\right),
\]
for all $p,\tilde{p}\in M$. Thus we can always apply Theorem \ref{GenGradient} to
the process
\[
\mathbf{X}_{t}=\left(B_{t},\int_{0}^{t}d_{M}\left(B_{s},p_0 \right)ds\right),
\]
where $B_{t}$ is Brownian motion on $M$.

\end{example}

\subsection{Iterated Kolmogorov diffusions}\label{5.3}

Our technique can also be applied in studying \emph{iterated Kolmogorov
diffusions}. These processes have been studied recently by S.~Banerjee and W.~Kendall in \cite{BanerjeeKendall2016a} and K.~Habermann in \cite{Habermann2018arxiv} as they provide a natural class of diffusions satisfying a weak H\"{o}rmander condition.

An iterated Kolmogorov diffusion is of the form $\mathbf{X}_{t}=\left(B_t, I_{1}(t), \dots, I_{n}(t)\right)$, where
\begin{align*}
I_{0}(t) & =\sigma\left(B_{t}\right),\\
I_{r}(t) & =\int_{0}^{t}I_{r-1}(s)ds,\,\,\,\text{ for }r=1,\dots,n,
\end{align*}
where $B_{t}$ is a Brownian motion on a manifold $M$ and $\sigma:M\to\mathbb{R}^k$
is $C_{\sigma}-$Lipschtiz. Let $P_{t}$ be the heat semigroup corresponding
to the diffusion $$\mathbf{X}_{t}=\left(B_t,I_{1}(t), \dots, I_{n}(t)\right).$$
Using an argument similar to the proof of Theorem \ref{GenGradient}, we get the following result.
\begin{theorem}
Let $M$ be a complete connected Riemannian manifold such that $\operatorname{Ric}(M)\geqslant K$
for some $K\in\mathbb{R}$. When $K=0$ and $f\in C_{0}^{\infty}\left(M\times\mathbb{R}^{k}\times\cdots\times\mathbb{R}^{k}\right)$ with
$f\left(p,\xi_{1},\dots,\xi_{n}\right),p\in M,\xi_{1},\dots,\xi_{n}\in\mathbb{R}^{k}$
we have the following gradient bound for the iterated Kolmogorov diffusion
semigroup $P_{t}$,
\[
\left\Vert \nabla_{p}P_{t}f\right\Vert ^{q}\leqslant P_{t}\left(\left(\left\Vert \nabla_{p}f\right\Vert +C_{\sigma}t\left\Vert \nabla_{\xi_{1}}f\right\Vert +\cdots+C_{\sigma}\frac{t^{n}}{n!}\left\Vert \nabla_{\xi_{n}}f\right\Vert \right)^{q}\right),
\]
for $q\geqslant1$. When $K\neq0$, we have
\[
\left\Vert \nabla_{p}P_{t}f\right\Vert ^{q}\leqslant P_{t}\left(\left(\left\Vert \nabla_{p}f\right\Vert +K_{1}(t)\left\Vert \nabla_{\xi_{1}}f\right\Vert +\cdots+K_{n}(t)\left\Vert \nabla_{\xi_{n}}f\right\Vert \right)^{q}\right),
\]
for $q\geqslant1$, where
\begin{align*}
K_{1}(t) & =C_{\sigma}\frac{1-e^{-Kt/2}}{K/2},\\
K_{r}(t) & =\int_{0}^{t}K_{r-1}(s)ds,\,\,\,\text{ for }r=2,\dots,n.
\end{align*}
\end{theorem}

\subsection{Heisenberg group}\label{5.4}

The Heisenberg group is the simplest nontrivial example of a sub-Riemannian manifold. The 3-dimensional Heisenberg group is $\mathbb{G}=\mathbb{R}^{3}$ with the group law defined by
\[
\left(x_{1},y_{1},z_{1}\right)\star\left(x_{2},y_{2},z_{2}\right):=\left(x_{1}+x_{2},y_{1}+y_{2},z_{1}+z_{2}+\frac{1}{2}\left(x_{1}y_{2}-x_{2}y_{1}\right)\right).
\]
The identity element is $e=\left(0,0,0\right)$ with the inverse given by $\left(x,y,z\right)^{-1}=\left(-x,-y,-z\right)$. We define the following left-invariant vector fields by
\begin{align*}
\mathcal{X} & :=\partial_{x}-\frac{y}{2}\partial_{z},\\
\mathcal{Y} & :=\partial_{y}-\frac{x}{2}\partial_{z},\\
\mathcal{Z} & :=\partial_{z}.
\end{align*}
The horizontal distribution is defined by $\mathcal{H}=\operatorname{span}\left\{ \mathcal{X},\mathcal{Y}\right\}$,
fiberwise. Vectors in $\mathcal{H}$ are said to be \emph{horizontal}.
We endow $\mathbb{G}$ with the sub-Riemannian metric $g\left(\cdot,\cdot\right)$
so that $\left\{ \mathcal{X},\mathcal{Y}\right\} $ forms an orthogonal
frame for the horizontal distribution $\mathcal{H}$. With this metric
we can define norms on vectors by $\left\Vert v\right\Vert =\left(g_{p}\left(v,v\right)\right)^{\frac{1}{2}}$
for $v\in\mathcal{H}_{p},p\in\mathbb{G}$. The Lebesgue measure on $\mathbb{R}^{3}$ is a Haar measure on the Heisenberg group. The distance associated to $\mathcal{H}$ is the Carnot-Carath\'{e}odory distance $d_{CC}$. The \emph{horizontal gradient }$\nabla_{\mathcal{H}}$ is a horizontal
vector field such that for any smooth $f:\mathbb{G}\rightarrow R$ we have that for all $X \in \mathcal{H}$
\[
g\left(\nabla_{\mathcal{H}}f,X\right)=X\left(f\right).
\]
The operator
\[
\Delta_{\mathcal{H}}=\frac{1}{2}\left(\mathcal{X}^{2}+\mathcal{Y}^{2}\right)
\]
is a natural sub-Laplacian for the Heisenberg as pointed out in \cite{AgrachevBoscainGauthierRossi2009, GordinaLaetsch2017} and also in \cite[Example 6.1]{GordinaLaetsch2016a}. Brownian motion on the Heisenberg group is defined to be the diffusion process $\left\{ B_{t}^{p}\right\} _{t \geqslant 0}$ starting at  $p=\left(x, y, z\right)\in\mathbb{R}^{3}$ whose infinitesimal generator is $\Delta_{\mathcal{H}}$. Explicitly the process is given by
\[
B_{t}^{p}=\left(B_{1}(t),B_{2}(t),z+\frac{1}{2}\int_{0}^{t}B_{1}(s)dB_{2}(s)-\frac{1}{2}\int_{0}^{t}B_{2}(s)dB_{1}(s)\right),
\]
where $\left(B_{1},B_{2}\right)$ is a Brownian motion starting at $\left(x,y\right)$.

Gradient bounds of Bakry-\'Emery type were studied for the Heisenberg group in \cite{BakryBaudoinBonnefontChafai2008, LiHong-Quan2006, DriverMelcher2005, Eldredge2010}. In particular, the $L^{1}$-gradient bounds for the heat semigroup have been proven first in \cite{LiHong-Quan2006} and also in \cite{BakryBaudoinBonnefontChafai2008}. As pointed out in \cite{Kuwada2010a}, Kuwada's duality between $L^{1}$-gradient bounds and $L^{\infty}$-Wasserstein  control shows that for each $t>0$, and $p, \widetilde{p} \in\mathbb{G}$, there exists a  coupling $\left(B_{t}^{p}, \widetilde{B}_{t}^{p}\right)$ of Brownian motions on the Heisenberg group such that
\begin{equation}\label{Heisenberg1}
d_{CC}\left(B_{t}^{p}, \widetilde{B}_{t}^{p}\right)\leqslant Kd_{CC}\left(p,\widetilde{p}\right),
\end{equation}
almost surely for some constant $K\geqslant1$ that does not depend on $p, \widetilde{p}, t$. We remark that in \cite{BonnefontJuillet2018}, the authors show that any coupling that satisfies \eqref{Heisenberg1} on $\mathbb{G}$ must be non-Markovian. This further highlights the need for more non-Markovian coupling techniques as in \cite{BanerjeeKendall2016a, BanerjeeGordinaMariano2018}.

Consider the Kolmogorov diffusion $\mathbf{X}_{t}=\left(B_{t}^{p},\xi+\int_{0}^{t}\sigma(B_{s}^{p})ds\right)$
on $\mathbb{G}\times\mathbb{R}^{3}$, where $\sigma:\mathbb{G}\to \mathbb{R}^3$ satisfies \eqref{Lip1} and let $P_{t}$ be the heat semigroup associated with $X_{t}$. Using a similar argument as in Lemma \ref{prop:Taylor} with the sub-Riemannian metric $g$ and the horizontal gradient $\nabla_{\mathcal{H}}$, we can get an estimate
\begin{align}
\left|f\left(p,\xi\right)-f\left(\widetilde{p},\xi\right)\right| & \leqslant\left\Vert\nabla_{\mathcal{H}}f(\widetilde{p},\xi)\right\Vert d_{CC}\left(p,\widetilde{p}\right)+\left\Vert\nabla_{\xi}f(\widetilde{p},\xi)\right\Vert d_{E}\left(\xi,\widetilde{\xi}\right)
\notag
\\
 & +C_{f}\left(d_{CC}\left(p,\widetilde{p}\right)+d_{E}\left(\xi,\widetilde{\xi}\right)\right)^2,
 \label{Heisenberg2}
\end{align}
for functions $f\in C_{0}^{\infty}\left(\mathbb{G}\times\mathbb{R}^{3}\right)$,
where $C_{f}\geqslant0$. The argument in Theorem \ref{GenGradient}
can be used to prove gradient bounds for $P_t$ when $B_t^p$ is a Brownian motion on a sub-Riemannian
manifold once we have a synchronous coupling and an estimate similar to \eqref{Heisenberg2}. Thus using \eqref{Heisenberg1} and \eqref{Heisenberg2} for the Heisenberg group we obtain the following result.
\begin{theorem}\label{HGrad}
For all $q\geqslant 1$ and $f\in C_{0}^{\infty}\left(\mathbb{G}\times\mathbb{R}^{3}\right)$,
\begin{equation}\label{HeisenbergGradientBound}
\left\Vert \nabla_{\mathcal{H}}P_{t}f\right\Vert ^{q}\leqslant K^q P_{t}\left(\left(\left\Vert \nabla_{\mathcal{H}}f\right\Vert +C_\sigma t\left\Vert \nabla_{\xi}f\right\Vert \right)^{q}\right).
\end{equation}
\end{theorem}
The best constant $K$ in \eqref{Heisenberg1}  and \eqref{HeisenbergGradientBound} is not known. The best known estimate for $K$ as of this writing is $K\geqslant \sqrt{2}$ (see \cite[Proposition 2.7]{DriverMelcher2005}). In fact the best constant $K$ is conjectured to be $\sqrt{2}$ in \cite[Remark 3.2]{BakryBaudoinBonnefontChafai2008}.

\begin{example}
Consider for $p=\left(x,y,z\right)\in\mathbb{G}$ the map $\sigma:\mathbb{G}\to\mathbb{R}^{3}$
defined by $\sigma\left(p\right)=\left(x,y,0\right)$ and the diffusion
$\mathbf{X}_{t}=\left(B_{t}^{p},\xi+\int_{0}^{t}\sigma\left(B_{s}^{p}\right)ds\right)$.
A straightforward computation shows that
\[
\sqrt{x^{2}+y^{2}}\leqslant d_{CC}\left(e,p\right),
\]
so that by the left-invariance of $d_{CC}$ we have that $\sigma$ is $1$-Lipschitz
in the sense of \eqref{Lip1}. Thus Theorem \ref{HGrad}
can be applied to $\mathbf{X}_{t}$.
\end{example}

\begin{acknowledgement}
The authors would like to thank Sayan Banerjee,  Bruce Driver and Tai Melcher for helpful discussions and insights. We would also like to thank two anonymous referees for their careful review of the paper and whose suggestions greatly improved the present paper.
\end{acknowledgement}

\bibliographystyle{amsplain}
\bibliography{KolmogorovDiffusions}

\end{document}